\documentclass[12pt, final]{amsart}
\usepackage{amsmath,amssymb}
\usepackage{amsthm}
\usepackage{amscd}
\usepackage{epsfig}
\usepackage{graphics}
\usepackage{graphicx}
\usepackage{colordvi}
\usepackage{mathrsfs} 
\usepackage[colorlinks]{hyperref}
\usepackage{dsfont}
\usepackage{mparhack}
\usepackage{xcolor,lipsum}

\oddsidemargin
0.2in \evensidemargin 0.2in\marginparwidth 0.4in

\newcommand{\1}{\mathds 1}

\newcommand{\F}{\mathscr F}

\newcommand{\bN}{\mathbb{N}}
\newcommand{\R}{{\mathbb{R}}}

\newcommand{\vep}{{\varepsilon}}

\newcommand{\lra}{\longrightarrow}

\newcommand{\lmt}{\longmapsto}

\newcommand{\ms}{\mathscr}
\renewcommand{\P}{{\mathbb P}}
\newcommand{\Pn}{\mathbb {P}^{(n)}}
\newcommand{\E}{{\mathbb E}}
\newcommand{\En}{\mathbb{E}^{(n)}}

\newcommand{\dxi}{{\widehat{\xi}}}
\newcommand{\hxi}{{\widehat{\xi}}}

\newcommand{\Q}{{\mathbb{Q}}}
\newcommand{\bQ}{{\mathbb{Q}}}

\newcommand{\pid}{\nu(\1)}
\newcommand{\piD}{\pi_{\rm diag}}
\newcommand{\piDn}{\pi^{(n)}_{\rm diag}}
\newcommand{\pidn}{\nu_n(\1)}

\newcommand{\gap}{\mathbf g}

\newcommand{\convdn}{\xrightarrow[n\to\infty]{{\rm (d)}}}

\newcommand{\pinmax}{\pi_{\rm max}^{(n)}}

\newcommand{\bpi}{\bar{\nu}}

\newcommand{\bZ}{\mathbb{Z}}
\newcommand{\tmix}{\mathbf t_{\rm mix}}
\newcommand{\tmixn}{\mathbf t_{\rm mix}^{(n)}}
\newcommand{\tmeet}{\mathbf t_{\rm meet}}
\newcommand{\tmeetn}{\mathbf t_{\rm meet}^{(n)}}

\newcommand{\cnkj}{ {{\sf C}^{(n)}_{k,j}}}

\newcommand{\cL}{\mathcal{L}}
\newcommand{\ds}{\displaystyle}
\begin{document}

\newtheoremstyle{slantthm}{10pt}{10pt}{\slshape}{}{\bfseries}{}{.5em}{\thmname{#1}\thmnumber{ #2}\thmnote{ (#3)}.}
\newtheoremstyle{slantrmk}{10pt}{10pt}{\rmfamily}{}{\bfseries}{}{.5em}{\thmname{#1}\thmnumber{ #2}\thmnote{ (#3)}.}

\newtheorem{thm}{Theorem}[section]
\newtheorem{prop}[thm]{Proposition}
\newtheorem{lem}[thm]{Lemma}
\newtheorem{cor}[thm]{Corollary}
\newtheorem{defi}[thm]{Definition}
\newtheorem{prob}[thm]{Problem}
\newtheorem{disc}[thm]{Discussion}
\newtheorem*{nota}{Notation}
\newtheorem*{conj}{Conjecture}
\theoremstyle{definition}
\newtheorem{rmk}[thm]{Remark}
\theoremstyle{slantrmk}
\newtheorem{eg}[thm]{Example}

\numberwithin{equation}{section}

\title{\bf On the Convergence of Densities of Finite Voter Models to the Wright-Fisher Diffusion}
\author{Yu-Ting\ Chen}
\address{Centre de Recherches Math\'ematiques and Universit\'e de Montr\'eal}
\thanks{Research of the first author was supported in part by the UBC Four Year Doctoral Fellowship and the CRM-ISM Postdoctoral Fellowship.} 

\author{Jihyeok\ Choi}
\address{Department of Mathematics\\
Syracuse University}
\thanks{Research of the second author was supported in part by a grant from the National Science Foundation.} 

\author{J. Theodore\ Cox}
\address{Department of Mathematics\\
Syracuse University}
\thanks{Research of the third author was supported in part by grants from the National Science Foundation and the Simons Foundation.}

\subjclass[2000]{Primary: 60K35, 82C22, Secondary: 60F05, 60J60}
\keywords{Wright-Fisher diffusion, voter model,
  interacting particle system,
  dual processes, semimartingale convergence theorem}

\begin{abstract}
We study voter models defined on large sets. 
Through a perspective emphasizing the martingale property of voter density processes,
we prove that in general, their convergence to the Wright-Fisher diffusion only involves certain averages of the voter models over a small number of spatial locations. This enables us to identify
suitable mixing conditions on the
underlying voting kernels, one of which may just depend on their eigenvalues in some contexts, to obtain the convergence of density processes.
Our examples show that these conditions are satisfied by a large class of voter models on growing finite graphs.
\end{abstract}

\maketitle
\markboth{CONVERGENCE OF VOTER DENSITIES TO WRIGHT-FISHER DIFFUSION}{YU-TING\ CHEN, JIHYEOK\ CHOI, AND J. THEODORE\ COX}

\section{Introduction}\label{sec:intro}
The goal of this work is to investigate
the convergence of density processes in finite voter models to the Wright-Fisher diffusion. This convergence gives a mean-field approximation for voter models, and 
is also closely related to the mean-field approximation of coalescence times for the associated dual Markov chains (cf. the recent work of Oliveira
\cite{O:MFC} and \cite{O:CT}). 
Earlier examples for
such convergence of density processes are few and include
the traditional mean-field models
and the voter models on $d$-dimensional tori for
$d\geq 2$ (cf. Cox \cite{C:CRW89}). 
In the present work,
we give mixing conditions on the underlying voting kernels which hold for a large class of finite voter models, and in particular generalize the earlier results.

We first introduce the class of voter models considered
throughout this paper.  (See Chapter~V of \cite{L:IPS} or Section 4.3 of \cite{L:CMP} for a general account of voter models.) Recall that for a finite set $E$, a $Q$-matrix $q$ is indexed by $x,y\in E$ and satisfies
\begin{align}\label{def:q}
q(x,y)\geq 0\;\;\quad\forall\; x\neq y\;\quad\mbox{and}\quad q(x)\equiv -q(x,x)=\sum_{y:y\neq x}q(x,y)
\end{align}
(see Chapter 2 of \cite{L:CMP}).
For such a pair $(q,E)$ with $q$ irreducible,
the associated continuous-time voter
model $(\xi_s)$ is the $\{0,1\}^E$-valued
Markov chain evolving according to the following rule. 
At independent exponential random times, the ``voter'' at site $x$ replaces its ``opinion'', which is $0$ or $1$, with that of another site chosen 
independently according to $q(x,\, \cdot\,)$ on $E\setminus\{x\}$.
More precisely, the voter model $(\xi_s)$ is the
pure-jump Markov process on $\{0,1\}^E$ with generator
\begin{align}\label{eq:gen}
\mathcal Lf(\xi)\equiv \sum_{x\in
  E}c(\xi,\xi^x)\big(f(\xi^x)-f(\xi)\big).
\end{align}
Here, for any configuration $\xi$, $\xi^x$ is obtained by switching the opinion of $\xi$ at $x$ to the opposite one and differs from $\xi$ only at this site, and the flip rate at which $\xi$ changes to $\xi^x$ is given by
\begin{align}\label{eq:flip}
c(\xi,\xi^x)=\sum_{y\in E}\left[\xi(x)\dxi(y)+\dxi(x)\xi(y)\right]q(x,y),
\end{align}
for $\dxi=1-\xi$. Hence, the $Q$-matrix $q$ can be interpreted as the {\bf voting kernel} of $(\xi_s)$. By allowing $q$ to be a general $Q$-matrix as in (\ref{eq:flip}),
we can consider the case that the total voting rates $q(x)$ (recall (\ref{def:q})) are site-dependent.

We consider in particular the density process
$\big(p_1(\xi_s)\big)$ of such a voter model, where
\begin{align}\label{eq:mc}
p_1(\xi)= \sum_{x\in E}\pi(x)\xi(x)
\end{align}
and $\pi$ is the unique stationary (probability) distribution of the irreducible $q$-Markov chain, that is the Markov chain with semigroup $(e^{tq};t\geq 0)$.
The simplest example arises from the mean-field model in which each $q(x,\,\cdot\,)$ is the uniform distribution on the set $E\setminus \{x\}$, and it is often called
the Moran model in population genetics.
In this
setting, 
$\pi$ is the uniform distribution on $E$, and it is straightforward to apply 
diffusion approximation to the density processes. More precisely, these processes, after
time-changes by suitable constants, converge in distribution in the
Skorokhod space to the Wright-Fisher diffusion as the
``population size'' $|E|$ tends to infinity.  Here, we recall that
the Wright-Fisher diffusion, denoted by 
\[
\left(Y,(\P_u)_{u\in [0,1]}\right) 
\]
throughout this paper, is a Markov process on $[0,1]$ which uniquely solves the well-posed martingale problem for
\begin{align}
\mathcal G\equiv \frac{1}{2}x(1-x)\frac{d^2}{dx^2}\label{eq:FW}
\end{align}
and initial condition $u$ for every $u\in [0,1]$. In particular, the Wright-Fisher diffusion is a continuous martingale with
predictable quadratic variation
\begin{equation}\label{WFPQ}
\langle Y\rangle_t = \int_0^t Y_s(1-Y_s) ds.
\end{equation} 
See Section 10.3 in \cite{EK:MP} for the convergence of these density processes and Chapter 4 in the same reference for martingale problems. 

For more realistic modelling, several
works consider finite voter models where the voting kernels
$q$ are defined by spatial structures, or more precisely by the transition kernels of (simple) random walks on graphs (see Chapter 14 in \cite{AF:MC}, \cite{C:CRW89}, Section 6.9 in
\cite{D:RGD}, \cite{DW}, \cite{O:MFC}, and
\cite{SR}). We note that in theoretical biology, such voter models play an important role in the study of evolutionary dynamics where
the use of general spatial structures for the underlying social networks of biological identities is essential (cf. \cite{OHLN}, \cite{CDP:VMP}, \cite{C:BC} and the references there).
Voter models in these contexts become harder to analyze, but the mean-field case mentioned above may still serve as an important example in 
their studies.

For density processes in spatial voter models, the work \cite{C:CRW89}
obtains a similar diffusion approximation on $d$-dimensional discrete tori for $d\geq 2$. It proves that if the
initial laws for voter models are Bernoulli product measures
with a constant density, then the density processes, again after suitable constant time-changes, converge
to the Wright-Fisher diffusion. We note that the voting kernels defining the
voter models in \cite{C:CRW89} are nearest-neighbor ones
allowing only ``local'' interactions, whereas
interactions in the mean-field case are defined by voters living in ``well-mixed" populations and are very different in nature. Hence, the
fact that the Wright-Fisher diffusion appears as the
diffusion limit in both cases suggests that this type of
diffusion approximation of density processes should occur in
some generality. More specifically, we will focus on the case as in \cite{C:CRW89} that the initial conditions are Bernoulli product measures.

To introduce our perspective on this question, we
restrict our attention to the simple case that 
\begin{align}
q=p-{\rm Id}_E\label{def:simple}
\end{align} 
for some symmetric probability matrix $p$ with zero diagonal
throughout this section.
Here, ${\rm Id}_E$ is the identity matrix indexed by elements of $E$, and such a $Q$-matrix $q$ arises when we consider the usual time-change of a discrete-time Markov chain with transition matrix $p$ by an independent rate-$1$ Poisson process (cf. Section 20.1 of \cite{LPW}). We will give in Section~\ref{sec:MR} our result for general irreducible voting kernels $q$, and more notation is required then. Now, the stationary distribution $\pi$ for a voting kernel $q$ of the form (\ref{def:simple}) is the uniform distribution, and
the density process 
\begin{align}\label{eq:denjump}
\mbox{$\big(p_1(\xi_s)\big)$ is
a martingale with jump size }\frac{1}{|E|}.
\end{align}
By introducing a
constant time-scale factor $\gamma>0$, the density process
has predictable quadratic variation
\begin{equation}
\langle p_1(\xi_{\gamma \cdot})\rangle_t=
\dfrac{2\gamma}{|E|}\int_0^tp_{10}(\xi_{\gamma s}) ds,\label{eq:pqv}
\end{equation}
where 
\begin{equation}\label{eq:sym-p10}
p_{10}(\xi) = \frac{1}{|E|}\sum_{x,y\in
  E}q(x,y)\xi(x)\hxi(y)
\end{equation}
is a weighted average of $(1,0)$ pairs in the configuration $\xi$.
See Proposition~\ref{prop:marts} for these properties of density processes.

This observation should readily reveal the similarity of the density process and the Wright-Fisher diffusion in terms of martingales, under the condition that the population size $|E|$ is large and
the predictable quadratic variation of the density process, a weighted average of $(1,0)$ pairs in $(\xi_{\gamma s})$ by (\ref{eq:pqv}), satisfies
\begin{align}
\langle p_1(\xi_{\gamma \cdot})\rangle_t \approx &
\int_0^tp_{1}(\xi_{\gamma s})[1-p_1(\xi_{\gamma_s})]ds\quad\mbox{as }|E|\lra\infty\label{eq:pqva}
\end{align}
(recall the predictable quadratic variation (\ref{WFPQ}) of the Wright-Fisher diffusion).
The mean-field case gives the simplest example satisfying this condition, since
\begin{align}
p_{10}(\xi)=\frac{|E|}{|E|-1}p_1(\xi)[1-p_1(\xi)],
\end{align}
and hence (\ref{eq:pqva}) holds plainly with $\gamma=|E|/2$. In general,
if we pass $|E|$ to infinity and $p_1(\xi_0)$ converges, then under (\ref{eq:pqva}) the
density processes should converge to a continuous martingale by (\ref{eq:denjump}) which
solves the well-posed martingale problem associated with the differential operator $\mathcal G$ in \eqref{eq:FW}.
In other words, the limiting object should be the Wright-Fisher process, and indeed, standard martingale arguments confirm this. See Section~\ref{sec:proof1} for the details, and also its last two paragraphs for the use of general initial conditions. 

We will formalize the condition (\ref{eq:pqva}) by considering the convergence \emph{in probability} of the differences
\begin{align}\label{eq:mfc0}
\langle p_1(\xi_{\gamma\cdot})\rangle_t-\int_0^tp_1(\xi_{\gamma_s})[1-p_1(\xi_{\gamma s})]ds
\end{align}
for any $t\in(0,\infty)$ and passing to the limit along a sequence of voter models, started with Bernoulli product measures with a constant density and defined by $\left(q^{(n)},E_n\right)_{n\in \Bbb N}$ with $|E_n|\lra\infty$, and a sequence of constant time scales $(\gamma_n)$. Our first main result in this paper shows that such convergence of the differences (\ref{eq:mfc0}) is in fact an \emph{equivalent} condition for the convergence of the voter densities toward the Wright-Fisher diffusion.
See Theorem~\ref{thm:main}.

Let us discuss how the method of moments in \cite{C:CRW89} can be applied to general finite voter models, and compare this method with the method of martingale problems stated above.
In \cite{C:CRW89}, the convergence of densities for
voter models on discrete tori toward the Wright-Fisher diffusion  was obtained by proving
that certain coalescence times of random walks are approximately sums of independent exponential variables
and then appealing to the method of moments via  the well-known duality between voter models and
coalescing Markov chains (see \cite{L:IPS} or (\ref{eq:duality}) below).
In fact, there are several connections between such almost exponentiality of coalescence times in terms of convergence in distribution and
the convergence to the Wright-Fisher diffusion of voter density processes, and they hold in general
(see Proposition~\ref{prop:ccmc} and Proposition~\ref{prop:ccmc2}).
To apply these connections, we note that 
the recent work of Oliveira in \cite{O:MFC} obtains the required asymptotic behavior of coalescence times for general Markov chains under Aldous's condition discussed below. This result can be readily used to get the mean-field behavior for one-dimensional marginals of the associated voter densities.
Nonetheless, in contrast to the method of moments, we believe
that the present approach by martingale problems gives greater insight into why the convergence to the Wright-Fisher diffusion should hold. It leads to an equivalent condition in terms of
the lower-order densities in (\ref{eq:mfc0}). 

The second main result of this paper is concerned with sufficient conditions for the convergence of the differences (\ref{eq:mfc0})
in terms of the underlying sequence of voting kernels
$q^{(n)}$.  By
Proposition~\ref{prop:norm} below, the convergence in probability of the differences (\ref{eq:mfc0}) for $q^{(n)}$ can be reinforced to
convergence in $L^2$-norm. Hence with duality, it can be shown that this convergence is equivalent to a
condition involving the coalescence times of \emph{four} $q^{(n)}$-Markov chains (recall (\ref{eq:pqv}) and see the remark below Proposition~\ref{prop:norm}). We give two simpler sufficient conditions for the convergence, and each involves just \emph{two} $q^{(n)}$-Markov
chains. These conditions result from the classical conditions for almost exponentiality of hitting times (see Aldous \cite{A:MCEHT} and Proposition 5.23 of Aldous and Fill \cite{AF:MC}), 
and carry the informal idea that the time for two independent chains to coalesce ``falls far behind'' the time for the chain to get close to stationarity.  See Theorem~\ref{thm:main2} for the precise formulations.
In formalizing the time to stationarity, while one of our two conditions (cf. Theorem~\ref{thm:main2} (i)) uses mixing times and also appears in \cite{O:MFC} for almost exponentiality of coalescence times, the other one
(cf. Theorem~\ref{thm:main2} (ii)) is based on spectral gaps and can be weaker, or more readily applied in some instances. On the other hand, by duality and our result for the convergence of voter densities, the latter condition can also serve as a weaker condition for the convergence in distribution of coalescence times to sums of independent exponential variables (Proposition~\ref{prop:ccmc}). See also Section 1.1 in \cite{O:MFC} for this issue when it comes to the stronger $L_1$-Wasserstein approximation of coalescence times.

As a final remark, we compare our results with the
convergence of the rescaled \emph{measure-valued} densities
of voter models on $\bZ^d$ to super-Brownian motions as
in Cox, Durrett and Perkins \cite{CDP} for $d\geq 2$ and to
a nonnegative solution of an SPDE as in Mueller and Tribe \cite{MT} for
$d=1$. These
voter models live on infinite
spatial structures which, after rescaling, converge in the natural way to
tractable geometric objects, namely Euclidean spaces of the
same dimension, and hence allow more detailed studies of the associated voter models. 
In our case, the analysis relies on the martingale property of densities, and we circumvent the issue of limiting spatial structures by turning to analytic conditions for almost exponentiality of coalescence times.

The paper is
organized as follows. In Section~\ref{sec:MR}, we present our main results for general finite voter models.
In Section~\ref{sec:ev}, we study some martingales associated with
a density process and use the duality equation for voter models to interpret these martingale properties in terms of coalescing Markov chains.
In Section~\ref{sec:paircoal}, 
we characterize the convergence of
the second moment of density processes in terms of the
asymptotic exponentiality of coalescence times. The results in this section are the core of our approach to obtain the convergence of density processes.
In Section~\ref{sec:proof1}, we study tightness of densities
and prove a general version (see Theorem~\ref{thm:main}) of the statement that the convergence of density processes to the Wright-Fisher diffusion is equivalent to the convergence in probability of the differences in (\ref{eq:mfc0}). As an application of this result, we prove in Section~\ref{sec:proof2} two sufficient conditions, each involving only
two independent $q$-Markov chains, for the convergence of voter densities (see
Theorem~\ref{thm:main2}).
In
Section~\ref{sec:cmc}, we discuss some connections between the convergence of coalescence times and the convergence of density processes, and the main results will be given below in
Proposition~\ref{prop:ccmc} and Proposition~\ref{prop:ccmc2}. 
Finally, Section~\ref{sec:eg} is devoted to a few examples to illustrate our sufficient conditions (see Theorem~\ref{thm:main2} and Corollary~\ref{cor:main}) for the convergence of density processes to the Wright-Fisher diffusion.

\section{Main results}\label{sec:MR}
From this section on, we consider voter models subject to irreducible $Q$-matrices (recall (\ref{def:q})) unless otherwise mentioned. 
We work with a sequence of irreducible $Q$-matrices
\[
(q^{(n)},E_n)_{n\in
  \mathbb{N}} 
  \]
with stationary (probability) distributions $(\pi^{(n)})$
whenever we study voter models on large sets, and a pair $(q,E)$ with stationary distribution $\pi$ otherwise. 
The voter models associated with such a sequence $(q^{(n)},E_n)$ started at Bernoulli
product measures $\mu_u$ with density $\mu_u(\xi(x)=1)=u$ are denoted by $\big((\xi_s),\P^{(n)}_{\mu_u}\big)$.
We will \emph{always}
assume that 
\[
|E_n|\lra\infty. 
\]
Whenever necessary, other quantities depending on $(q^{(n)},E_n)$ will carry subscripts `$n$' or superscripts `$(n)$'.  

We start with our result for the equivalent condition of the convergence of voter densities to the Wright-Fisher diffusion.
Now, for any pair $(q,E)$, the associated density process $\big(p_1(\xi_{\gamma t})\big)$ for $\gamma>0$ is a martingale with jump size bounded above by $\ds\max_{x\in
  E}\pi(x)$, and its predictable quadratic variation 
takes a more general form than (\ref{eq:pqv}) which is for the simpler case (\ref{def:simple}). To state the formula for the general case, we set up some notation. 
Introduce the following measures on the product space $E\times E$ induced by $\pi$ and $q$:
\begin{align}
\nu(x,y)\equiv &\pi(x)^2q(x,y)\1_{x\neq y},\label{eq:piD}\\
\bar{\nu}(x,y)\equiv &\nu(x,y)\big / \nu(\1).\label{eq:piD1}
\end{align}
In addition, set
$p_{10}(\xi)$ and
$p_{01}(\xi)$ as the $\bpi $-weighted averages of the ordered pairs
$(1,0)$ and $(0,1)$, respectively, in the configuration
$\xi$,  given by
\begin{align}
p_{10}(\xi)=&\sum_{x,y\in E} \bpi(x,y)\xi(x)\dxi(y),\label{eq:wa1-1}\\
p_{01}(\xi)=&\sum_{x,y\in E} \bpi(x,y)\dxi(x)\xi(y)\label{eq:wa1-2}.
\end{align}
Then
\begin{align}\label{eq:psvp}
  \langle p_1(\xi_{\gamma \cdot})\rangle_t=
  \gamma\pid\int_0^t\left[p_{10}(\xi_{\gamma
      s})+p_{01}(\xi_{\gamma s})\right]ds
\end{align}
(see Proposition~\ref{prop:marts} below).
Note that if $q$ is of the particular form (\ref{def:simple}), then $\pid=1/|E|$, both $p_{10}(\xi)$ and $p_{01}(\xi)$
agree with the right-hand side of \eqref{eq:sym-p10}, and the right-hand sides of
\eqref{eq:pqv} and \eqref{eq:psvp} are equal.

Below we use $\convdn$ to denote convergence in distribution and write
\[
\piD=\sum_{x\in E}\pi(x)^2.
\]
\medskip

\begin{thm}\label{thm:main}
Let $u\in (0,1)$ and let $(\gamma_n)$ be a sequence of strictly positive constants.
Assume that
\begin{align}\label{eq:diag}
\lim_{n\to\infty}\piDn=0.
\end{align}
Then the convergence of density processes
\begin{align}
\big(p_1(\xi_{\gamma_n\cdot}), \P^{(n)}_{\mu_u}\big)
\convdn  (Y,\P_u)\label{eq:main1}
\end{align}
under the Skorokhod $J_1$-topology for c\`adl\`ag functions holds if and
only if the following {\bf mean-field condition} holds: for
any $T\in (0,\infty)$,
\begin{align}
\begin{split}\label{eq:mfc}
&\gamma_n\,\pidn\int_0^T[p_{10}(\xi_{\gamma_ns})+p_{01}(\xi_{\gamma_ns})]ds\\
&\hspace{2cm}-\int_0^Tp_1(\xi_{\gamma_ns})[1-p_1(\xi_{\gamma_ns})]ds
\convdn 0. 
\end{split}
\end{align}
\end{thm}
\medskip 

We will show in Section~\ref{sec:paircoal} below (see Theorem~\ref{thm:2ndmom}) that the condition (\ref{eq:diag}) is in fact necessary for (\ref{eq:main1}).

Next, we discuss our second main result which gives sufficient conditions for the mean-field condition (\ref{eq:mfc}).
We need some notation concerning the mixing of the $q$-Markov chain.
Let $(q_t)=(e^{tq})$ be the semigroup of the
$q$-Markov chain on $E$, and
$d_E$ be the maximal total variation distance 
\begin{equation}\label{eq:dedef}
d_E(t)=\max_{x\in E}\left\|q_t(x,\cdot)-
\pi(\cdot)\right\|_{\rm TV},
\end{equation}
where $\|\cdot\|_{\rm TV}$ refers to the total variation distance. Note that $d_E(t)$ is always finite.
We recall that the mixing time 
\begin{align}\label{def:tmix}
\mathbf t_{\rm mix}=\inf\left\{t\geq
  0:d_E(t)\leq\frac{1}{2e}\right\}<\infty 
\end{align}
provides, informally speaking, one measurement of the time for
the one-dimensional marginals to get
close to the equilibrium 
distribution $\pi$.  An alternative for this purpose for the $q$-Markov chain is the associated relaxation time $\gap^{-1}$, where $\gap\in (0,\infty)$ is  the spectral gap and is the second smallest eigenvalue of $-q$. We refer to \cite{AF:MC}
and \cite{LPW} for standard properties of spectral gaps and their connections with mixing times (the arguments there can be adapted in a straightforward manner to the context of Markov chains defined by general $Q$-matrices according to the setup in Section 1.1 of \cite{CLR}).
In particular, we note that ${\mathbf g}^{-1}\le
\tmix $. 

Next, let $M_{U,U'}$ be
the meeting time of two independent $q$-Markov chains with
semigroup $(q_t)$ started at spatial locations
$(U,U')$, where the sites $U$ and $U'$ are independent and
distributed according to $\pi$. We define the
expected meeting time to be
\begin{equation}\label{tmeet}
\tmeet =  \mathbf{E}[M_{U,U'}] .
\end{equation}

\medskip

\begin{thm}\label{thm:main2}
For each $n\in \Bbb N$,  let $\mathbf g_n$, $\tmixn$ and $\tmeetn$ be the spectral
  gap, mixing time and expected meeting time of the
  $q^{(n)}$-Markov chain, respectively.
  In addition, we put
\begin{align*}
\ds\pinmax =& \max\{\pi^{(n)}(x);x\in E_n\},\\
q_{\max}^{(n)}=&\max\{q^{(n)}(x);x\in E_n\}.
\end{align*}  
(Recall that the voting rates $q^{(n)}(x)$ are defined in (\ref{def:q}).)
  Suppose that either of the
  following conditions is satisfied:
\begin{enumerate}
\item [\rm (i)] $\ds\lim_{n\to\infty}\piDn=0\quad\mbox{ and }\quad
  \displaystyle \lim_{n\to\infty} 
\frac{\tmixn}{\tmeetn}=0$,
\item [\rm (ii)] the $q^{(n)}$-Markov chains are
  reversible and satisfy, 
\begin{align}\label{eq:(ii)}
&\ds\lim_{n\to\infty}\piDn=0\quad\mbox{and}\quad \displaystyle \lim_{n\to\infty}
\dfrac{\log\big(e\vee\tmeetn\pinmax q_{\max}^{(n)}\big)}{\gap_n\tmeetn}=0.
\end{align}
\end{enumerate}
Then for all $u\in[0,1]$, 
(\ref{eq:mfc}) holds with
$\gamma_n= \tmeetn$, and consequently, \eqref{eq:main1} holds. 
\end{thm}

\medskip

Let us make some observation for the condition (ii) of Theorem~\ref{thm:main2}.
From an inequality (see (\ref{eq:meetlowerbound})) proved later on, we have 
\begin{align}\label{eq:tmeetnlbd}
\tmeetn \pinmax q_{\max}^{(n)}\geq \frac{(1-\piDn)^2}{4}.
\end{align}
Also, it is plain that 
\begin{align}\label{eq:piDnpinmax}
\lim_{n\to\infty}\piDn=0\Longleftrightarrow \lim_{n\to\infty}\pinmax=0.
\end{align}
Hence if the voting rates $\big(q^{(n)}(x);x\in E_n\big)$ are uniformly bounded and $\lim_{n\to\infty}\piDn=0$, then $\lim_{n\to\infty}\tmeetn=\infty$, and moreover,
$\tmeetn$ has order at least $(\pinmax)^{-1}$. This, applied to the second part of (\ref{eq:(ii)}), gives the following.

\begin{cor}\label{cor:main}
If the Markov chains defined by $(q^{(n)},E_n)$ are reversible and satisfy $\ds\lim_{n\to\infty}\piDn=0$,
\begin{align*}
\begin{split}
&\limsup_{n\to\infty}\max_{x\in E_n}q^{(n)}(x)<\infty,\quad\mbox{and}\quad\liminf_{n\to\infty}\gap_n>0,
\end{split}
\end{align*}
then the same conclusions of Theorem~\ref{thm:main2} hold. In particular, these conditions hold when $q^{(n)}=p^{(n)}-{\rm Id}_{E_n}$ for symmetric
probability matrices $p^{(n)}$ (not necessarily with zero diagonals), and
the Markov chains defined by $(q^{(n)},E_n)$ satisfy $\liminf_{n\to\infty}\gap_n>0$.
\end{cor}

If the sequence $\big(\tmeetn \pinmax q_{\max}^{(n)}\big)_{n\in \Bbb N}$ is bounded above, then plainly the
second condition in (\ref{eq:(ii)}) reduces to
\begin{align}\label{gapmeet}
\lim_{n\to\infty}\gap_n\tmeetn =\infty.
\end{align}
This is the condition suggested by Aldous and Fill on almost exponentiality of hitting times in \cite{AF:MC}, for the particular case of the first meeting time of two independent $q$-Markov chains (see also Section 1.1 in \cite{O:MFC}). Moreover, if $q^{(n)}=p^{(n)}-{\rm Id}_{E_n}$ for a probability matrix $p^{(n)}$ and the matrices $p^{(n)}$ satisfy sufficient symmetry (see Chapter 7 in \cite{AF:MC} for the notion of \emph{symmetric chains} and note that it is stronger than requiring $p^{(n)}(x,y)=p^{(n)}(y,x)$ for any $x,y$), then 
$2\tmeet^{(n)}$ is equal to the so-called \emph{random target time} and so can be expressed explicitly in terms of the eigenvalues of $p^{(n)}$ (Section 4.2 in \cite{AF:MC}). In this case, the condition (\ref{gapmeet}) only involves the eigenvalues of $-q^{(n)}$.

\begin{rmk}
One notion of ``transience''
  (respectively, ``recurrence'') for a sequence of finite
  Markov chains (see Section~15.2.3 in \cite{AF:MC}) is essentially
  that the sequence $\big( \tmeetn\pinmax q^{(n)}_{\max}\big)_{n\in \Bbb N}$ be bounded above
  (respectively, tend to infinity).  See Remark~\ref{rmk:rec_trans} for more details on this terminology. Theorem~\ref{thm:main2} applies in both cases. 
In fact, we use considerably more delicate arguments in the present proof of Theorem~\ref{thm:main2}, in order to take into account  
the recurrent case as well. \qed
\end{rmk}

Our last results concern coalescence times of Markov chains.
Suppose again
that we have a sequence of irreducible $Q$-matrices
$(q^{(n)},E_n)$, with stationary distributions $(\pi^{(n)})$. For a given $n$, let $U_1,U_2,\dots$ be i.i.d. with
distribution $\pi^{(n)}$. Let $(\hat X^x_t,x\in E_n)$ be a
system of coalescing $q^{(n)}$-Markov chains, with $\hat X^x_0=x$, independent of the $U_i$'s. This means that the
$q^{(n)}$-Markov chains $\hat X^x$ move independently until they
meet, at which time they coalesce and move together.
Define the coalescence times
\[
\cnkj = \inf\{t>0; |\{\hat X^{U_1}_t,\dots,\hat X^{U_k}_t\}|=j\},
\quad
1\le j\leq k\le |E_n|,
\]
and let $Z_2,Z_3,\dots$ be independent exponential random
variables with 
$\mathbf{E}[Z_j]=1/\binom{j}{2}$. In the
mean-field case, it is well-known and
easy to check that with $\gamma_n=|E_n|/2$,
\begin{equation}\label{eq:coalconv}
\frac{\cnkj}{\gamma_n} \convdn \sum_{i=j+1}^kZ_i,  \quad
1\le j<k<\infty.
\end{equation}
(See Chapter 14 in \cite{AF:MC}.)
In fact, this convergence is an easy consequence of the convergence of
voter model densities to the Wright-Fisher diffusion.

\begin{prop}
\label{prop:ccmc} If \eqref{eq:main1} holds,
  then so does \eqref{eq:coalconv}. In 
  particular, if either of the conditions of
  Theorem~\ref{thm:main2} hold, then so does
  \eqref{eq:coalconv} with $\gamma_n=\tmeetn$.
\end{prop}

We refer the readers to \cite{O:MFC} and
\cite{O:CT} for recent results on the almost exponentiality of Markov chain hitting times
of general sets, and in particular, of Markov chain coalescence times.  
These results give the convergence in (\ref{eq:coalconv}) with explicit convergence rates under slightly different
conditions than the ones we give here. 
Remarkably,
the convergence of the ``full'' coalescence times $\hat{\sf C}^{(n)}_1$ of $\{\hat
X^x;x\in E_n\}$ is also obtained in \cite{O:MFC}, where
\[
\hat{\sf C}^{(n)}_j=\inf\{t\geq 0;|\{\hat{X}^{x}_t;x\in E_n\}|=j\},\quad 1\leq j\leq |E_n|.
\] 
In this direction, we also have Proposition~\ref{prop:ccmc2} below, which interprets the convergence of full coalescence times in terms of the convergence of voter densities to the Wright-Fisher diffusion. 

\begin{prop}\label{prop:ccmc2} Let $\tau_1^{(n)}$ denote the first hitting time of $1$ by the density process $\big(p_1(\xi_{\gamma_n t})\big)$, and
$\tau^{Y}_1$ the first hitting time of $1$ by the Wright-Fisher diffusion $(Y_t)$.
Then the following convergences are equivalent:
\begin{align}
\left(\frac{ \tau_1^{(n)}}{\gamma_n},\P^{(n)}_{\mu_u} \right)&\convdn\left( \tau^{Y}_1,\P_u\right),\quad \forall\;u\in [0,1],\label{eq:full1}\\
\quad\frac{\hat{{\sf C}}^{(n)}_{j}}{\gamma_n}&\convdn \sum_{i=j+1}^\infty Z_i,\quad \forall\; j\in \Bbb N.\label{eq:full2}
\end{align}
\end{prop}

We note that the convergence (\ref{eq:full1}) does not follow immediately from the weak convergence of density processes since first hitting times are in general not continuous with respect to the Skorokhod $J_1$-topology. To see this, we may reinforce the convergence (\ref{eq:main1}) to almost-sure convergence in the Skorokhod $J_1$-topology by the Skorokhod representation (see \cite{EK:MP}). Then, for example,
the approximating density processes $\big(p_1(\xi_{\gamma_n\cdot}),\P^{(n)}_{\mu_u}\big)$ may ``linger'' very close to the absorbing state $1$ for long periods of time before getting absorbed at $1$, while the limiting process $(Y_t)$ has already reached $1$. Hence, (\ref{eq:full1}) rules out this lingering behavior of the density process $\big(p_1(\xi_{\gamma_n\cdot}),\P^{(n)}_{\mu_u}\big)$ for all large $n$ in particular.

\section{Martingale property and duality}\label{sec:ev}
Fix a Markov chain defined by $(q,E)$ with stationary distribution $\pi$, and consider the corresponding voter model
$(\xi_t)$. Recall the definition (\ref{eq:mc}) of $p_1$, and set 
\[
p_0\equiv 1-p_1.
\]
In this section, we identify some martingales associated with the density process $\big(p_1(\xi_t)\big)$ and then resort to the duality equation for voter models (see (\ref{eq:duality}) below) for their interpretations in terms of coalescing Markov chains.

\begin{prop}\label{prop:marts} For any initial
  configuration $\xi\in\{0,1\}^E$, all of the following three processes are
  $\P_\xi$-martingales:
\begin{enumerate}
\item [\rm (i)] $\big(p_1(\xi_t)\big)$
\item [\rm (ii)] $\ds \left(p_1(\xi_t)p_0(\xi_t)+\pid \int_0^t
[p_{10}(\xi_s)+p_{01}(\xi_s)]ds\right)$
\item [\rm (iii)] $\ds\left(p^2_1(\xi_t)- \pid \int_0^t
[p_{10}(\xi_s)+p_{01}(\xi_s)]ds\right)$.
\end{enumerate}
\end{prop}

\begin{proof}
Recall that the generator $\mathcal L$ and the flip rates of the voter model $(\xi_t)$ are given by (\ref{eq:gen}) and (\ref{eq:flip}), respectively. In the following, we will show
\begin{align}
\mathcal Lp_1&\equiv 0,\label{eq:Lp1}\\
\mathcal L(p_1p_0)&\equiv -\pid (p_{10}+p_{01}).\label{eq:p0p1}
\end{align}
Then our assertions for the processes in (i) and (ii) follow from these and a standard result of Markov processes. 
The fact that the process in (iii) is a martingale then follows from the analogous properties of the processes in (i) and (ii), since
$p_1^2 = p_1- p_1p_0$.

We first show (\ref{eq:Lp1}). Plainly
\begin{align}\label{eq:p1diff}
p_1(\xi^x)-p_1(\xi)=
\pi(x)\big[\hxi(x)-\xi(x)\big], 
\end{align}
and thus by (\ref{eq:gen}) we get
\begin{align*}
\mathcal Lp_1(\xi)=&\sum_{x\in E}\dxi(x)\sum_{y\in E}
  \xi(y)q(x,y)\pi(x)-\sum_{x\in E}
\xi(x)\sum_{y\in E}\dxi(y)q(x,y)\pi(x)\\ 
=&\sum_{x,y\in E} \xi(y)\pi(x)q(x,y)-\sum_{x,y\in E}\xi(x)\xi(y)\pi(x)q(x,y)\\
&-\sum_{x,y\in E}\xi(x)\pi(x)q(x,y)+\sum_{x,y\in E} \xi(x)\xi(y)\pi(x)q(x,y)
=0,
\end{align*}
because $\sum_{y\in E} q(x, y) = 0$ and $\sum_{x\in E}\pi(x)q(x, y) = 0$ for all $y \in E$. Hence, (\ref{eq:Lp1}) follows, and the density process $\big(p_1(\xi_t)\big)$ is a martingale.

Next, to show (\ref{eq:p0p1}), we note that for any $x\in E$,  
\begin{align*}
p_1(\xi^x)p_0&(\xi^x)-p_1(\xi)p_0(\xi)\\
&=[p_1(\xi^x)-p_1(\xi)]\cdot [p_0(\xi^x)-p_0(\xi)]+p_0(\xi)[p_1(\xi^x)-p_1(\xi)]\\
&\qquad +p_1(\xi)[p_0(\xi^x)-p_0(\xi)]\\
&=-\pi(x)^2+p_0(\xi)[p_1(\xi^x)-p_1(\xi)]
+p_1(\xi)[p_0(\xi^x)-p_0(\xi)].
\end{align*}
where we have used (\ref{eq:p1diff}) and the analogue
$p_0(\xi^x)-p_0(\xi)=\pi(x)[\xi(x)-\hxi(x)]$ in the last line.
Since $\cL p_1=\cL p_0=0$, the last equality implies that
\begin{align*}
\sum_{x\in E}c(x,\xi)&[p_1(\xi^x)p_0(\xi^x)-p_1(\xi)p_0(\xi)]\\
&=-\sum_{x,y\in
  E}[\xi(x)\hxi(y)+\hxi(x)\xi(y)]q(x,y)\pi(x)^2\\
&= -\pid [p_{10}(\xi) + p_{01}(\xi)],
\end{align*}
where the last equality follows from the definitions (\ref{eq:piD})--(\ref{eq:wa1-2}). This gives (\ref{eq:p0p1}), and our assertion for (ii) is proved.
The proof is complete.
\end{proof}

\medskip

Recall that $\mu_u$ denotes the Bernoulli product measure on $E$ with density $\mu_u(\xi(x)=1)=u$. 
\begin{cor}\label{cor:marts}
For any $\gamma,t\in (0,\infty)$ and initial configuration $\xi\in\{0,1\}^E$, the martingale $\big(p_{1}(\xi_{\gamma t})\big)$ under $\P_\xi$ has
 predictable quadratic variation process
\begin{align}
  \langle p_1(\xi_{\gamma \cdot})\rangle_t&=
  \gamma\pid\int_0^t\left[p_{10}(\xi_{\gamma
      s})+p_{01}(\xi_{\gamma s})\right]ds.\label{eq:psvp2}
\end{align}
Also for any $u\in [0,1]$, we have
\begin{equation}\label{eq:time0}
\E_{\mu_u}[p_1(\xi_{\gamma t})p_0(\xi_{\gamma t
})]=u(1-u)(1-\piD)-\gamma\,
\pid\int_0^t \E_{\mu_u}[p_{10}(\xi_{\gamma
  s})+p_{01}(\xi_{\gamma s})]ds.
\end{equation}
\end{cor}
\begin{proof}
The equation (\ref{eq:psvp2}) follows readily from Proposition~\ref{prop:marts} for the process in (iii) and the standard characterization of predictable quadratic variations (cf. \cite{JS:LT}). Similarly, by (ii) in the same proposition, we have
\begin{align}
\E_{\xi}[p_1(\xi_{\gamma t})p_0(\xi_{\gamma t
})]&=p_1(\xi)p_0(\xi)-\gamma\,
\pid\int_0^t \E_\xi[p_{10}(\xi_{\gamma
  s})+p_{01}(\xi_{\gamma s})]ds,\label{eq:marts2}
\end{align}
and so a randomization of the initial configuration $\xi$ by $\mu_u$ leads to (\ref{eq:time0}).
\end{proof}

The rest of this section is devoted to interpreting the above results by coalescing Markov chains, and now
we recall duality. 
Using the coalescing Markov chains
$(\hat{X}^x,x\in E)$ introduced in Section~\ref{sec:MR}, we can formulate 
the duality equation for voter models  (see Chapter~V of
\cite{L:IPS} or Section 4.3 of \cite{L:CMP}) as
\begin{equation}\label{eq:duality}
\E_\eta\left[\prod_{x\in F}\xi_t(x)\right]=\mathbf
E\left[\prod_{x\in F}\eta\big(\hat{ X}^x_t\big)\right]\quad \forall
\;\eta\in\{0,1\}^E, \;t\in\R_+ 
\end{equation}
for any nonempty subset $F$ of $E$.
The readers will see later on that the duality formula becomes particularly tractable for a voter model with initial
law $\mu_u$.

We will make frequent use of a special case of
\eqref{eq:duality} stated as follows. For convenience, let $(X^x_t,x\in E)$ be another system of $q$-Markov chains with $Q$-matrix
$q$ and $X^x_0=x$, but now consist of independent chains.
We define the first meeting
times of $X^x$ and $X^y$ by
\[
M_{x,y}= \inf\{t\ge 0: X^x_t=X^y_t\}, \quad x,y\in E.
\]
Then \eqref{eq:duality} implies
\begin{equation}\label{eq:duality2}
\E_\xi\big[\xi_t(x)\hxi_t(y)\big] = \mathbf{E}\big[
\xi(X^x_t)\hxi(X^y_t); M_{x,y}>t\big].
\end{equation}

Next, we recall that $(U,U')$ has law $\pi\otimes\pi$, and now
introduce $(V,V')$ with law
\begin{align}\label{eq:distVV}
\mathbf P(V=a,V'=b)&\equiv \bpi(a,b), \quad a,b\in E
\end{align}
(recall the definition of $\bpi$ from (\ref{eq:piD1})).
We assume, in addition, that these random elements $(U,U')$ and $(V,V')$ are
independent of the system $(X^x;x\in E)$.

\begin{prop}\label{prop:pairapprox}
For any $\gamma, t> 0$ and initial configuration $\xi\in \{0,1\}^E$,
\begin{align}
\E_{\xi}[p_1(\xi_{\gamma t})p_0(\xi_{\gamma_t})]&=
\mathbf E[\xi(X^U_{\gamma t})\hxi(X^{U'}_{\gamma t}) ;
M_{U,U'}>\gamma t]\label{eq:p1p0xi},\\
\E_{\xi}[p_{10}(\xi_{\gamma t})]&=
\mathbf E[\xi(X^V_{\gamma t})\hxi(X^{V'}_{\gamma t});
M_{V,V'}>\gamma t]\label{eq:p10xi},\\
\E_{\xi}[p_{01}(\xi_{\gamma t})]&=
\mathbf E[\hxi(X^V_{\gamma t})\xi(X^{V'}_{\gamma t});
M_{V,V'}>\gamma t]\label{eq:p01xi} .
\end{align}
\end{prop}

\begin{proof}
By the duality equation \eqref{eq:duality2} and the
definitions of $p_1,p_0$ and $(U,U')$, we have
\begin{align*}
\E_{\xi}[p_1(\xi_{\gamma t})p_0(\xi_{\gamma
  t})]&= \sum_{x,y\in
  E}\pi(x)\pi(y)\mathbf{E} [\xi(X^x_{\gamma t})
\hxi(X^y_{\gamma t}); M_{x,y}>t]\\
&= \mathbf E[\xi(X^U_{\gamma t})\hxi(X^{U'}_{\gamma t}); 
M_{U,U'}>\gamma t],
\end{align*}
which proves \eqref{eq:p1p0xi}. The equations
\eqref{eq:p10xi} and \eqref{eq:p01xi} can be derived in the
same fashion by using the definition (\ref{eq:distVV}) of $(V,V')$.
\end{proof}

We point out that (\ref{eq:p1p0xi})--(\ref{eq:p01xi}) are
 closely related to the tail distributions of some particular meeting times. 
By \eqref{eq:p10xi} and
  \eqref{eq:p01xi}, we have
\begin{equation}\label{eq:p1001bnd}
\sup_{\xi\in \{0,1\}^E} \E_{\xi}[p_{10}(\xi_{\gamma t})+p_{01}(\xi_{\gamma t})]
\le 2 \mathbf P(M_{V,V'}>\gamma t).
\end{equation}
Moreover, if we start the voter model with the product measure $\mu_u$ for $u\in[0,1]$, then 
Proposition~\ref{prop:pairapprox} implies
\begin{equation}\label{eq:p1p0mu}
\E_{\mu_u}[p_1(\xi_{\gamma
  t})p_0(\xi_{\gamma_t})]=u(1-u)\mathbf P(M_{U,U'}>\gamma
t)
\end{equation}
and
\begin{equation}\label{eq:p10mu}
\E_{\mu_u}[p_{10}(\xi_{\gamma
  t})]=\E_{\mu_u}[p_{01}(\xi_{\gamma t})]=u(1-u)\mathbf
P(M_{V,V'}>\gamma t) .
\end{equation}

As a particular application of (\ref{eq:p1p0mu}) and (\ref{eq:p10mu}), we give simple proofs for some known results in Markov chain theory in Corollary~\ref{cor:MUV} below (see Section 5.3 of Chapter 3 in \cite{AF:MC}).
\medskip

\begin{cor}\label{cor:MUV}
The tail distributions of $M_{U,U'}$ and $M_{V,V'}$ are related by the formula: for any $\gamma, t>0$,
\begin{equation}\label{eq:MUV}
\mathbf{P}(M_{U,U'}>\gamma
t) =1-\piD -2\gamma \pid\int_0^t \mathbf P(M_{V,V'}>\gamma s)ds.
\end{equation}
Moreover, we have
\begin{align}
\mathbf E[M_{V,V'}]&=\frac{1-\piD}{2\pid},\label{eq:MVVmom}\\
\mathbf E[M_{U,U'}]&=\pid\mathbf E[M_{V,V'}^2].\label{eq:MUUmom}
\end{align}
\end{cor}
\begin{proof}
We start with \eqref{eq:MUV}. 
If we fix $u\in(0,1)$, and plug \eqref{eq:p1p0mu} and
\eqref{eq:p10mu} into \eqref{eq:time0}, then 
cancelling the factor $u(1-u)$ gives
\eqref{eq:MUV}. We remark that \eqref{eq:MUV} can be alternatively
derived by a standard Markov chain ``last time''
decomposition (see Section A.2 of \cite{CDP}), and leave the details to the readers.

We then consider the two equalities (\ref{eq:MVVmom}) and (\ref{eq:MUUmom}).
Since $q$ is irreducible, the meeting time $M_{x,y}$ is
finite a.s. for any $x,y\in E$.  Thus,  
by setting $\gamma=1$ and passing $t\lra\infty$ in the identity 
\eqref{eq:MUV}, we deduce \eqref{eq:MVVmom}.  
To obtain the second equality (\ref{eq:MUUmom}), we set
$\gamma=1$ and integrate both sides of \eqref{eq:MUV}:
\begin{align*}
\mathbf E[M_{U,U'}]&=2\pid\int_0^\infty\left( \frac{1-\piD}{2\pid}-\int_0^t \mathbf P(M_{V,V'}>s)ds\right)dt\\
&=2\pid\int_0^\infty \int_t^\infty \mathbf P(M_{V,V'}>s)dsdt\\
&=2\pid\int_0^\infty s\mathbf P(M_{V,V'}>s)ds\\
&=2\pid \frac{\mathbf E[M_{V,V'}^2]}{2}=\pid\mathbf E[M^2_{V,V'}],
\end{align*}
where (\ref{eq:MVVmom}) is used in the second equality below.
We have proved (\ref{eq:MUUmom}). The proof is complete.
\end{proof}

\begin{rmk} 
(1) Some useful consequences of Corollary~\ref{cor:MUV} are
the following. First, \eqref{eq:MVVmom} and Markov's inequality imply that
for any $\gamma,t>0$, 
\begin{align}
\label{eq:VVtailbound}
&2\gamma \pid \mathbf{P}(M_{V,V'}>\gamma t) \le
\dfrac{1-\piD}{t}.
\end{align}
Second, passing $t\lra\infty$ in \eqref{eq:MUV}, we obtain
\begin{equation}
 \label{eq:MUV2}
2\gamma \pid\int_0^\infty \mathbf P(M_{V,V'}>\gamma s)ds \le 1.
\end{equation}
Finally, from \eqref{eq:MVVmom},
\eqref{eq:MUUmom} and the Cauchy-Schwartz inequality we obtain a useful
lower bound of $\tmeet=\mathbf E[M_{U,U'}]$:
\begin{equation}\label{eq:meetlowerbound}
\tmeet \ge \frac1{\pid} \,\left( \frac{1-\piD}2\right)^2.
\end{equation}
See also Section 5.1 of \cite{AF:MC} for a similar inequality.
\medskip

\noindent (2) If $q$ is of the form (\ref{def:simple}) for a symmetric probability matrix $p$ with zero diagonal, then $\pid =
1/|E|$ and $\mathbf{P}\big((V,V')=(a,b)\big)=\pi(a)q(a,b)$. In this
case, \eqref{eq:MVVmom} and \eqref{eq:MUUmom} reduce to 
\[
\mathbf{E}[M_{V,V'}]=
\frac{|E|-1}{2}\quad\mbox{ and }\quad
\mathbf E[M^2_{V,V'}]=|E|\cdot \mathbf{E} [M_{U,U'}],
\]
respectively.
\qed
\end{rmk}

\section{Pairwise coalescence times}\label{sec:paircoal}
Throughout this section we take an arbitrary sequence of
irreducible Markov chains defined by $Q$-matrices $(q^{(n)},E_n)_{n\in \bN}$.
With $\pi^{(n)}$ being the stationary distribution of $q^{(n)}$, we write
\begin{align*}
\piDn=&\sum_{x\in E_n}\pi^{(n)}(x)^2,\\
\pidn(x,y)\equiv &\pi^{(n)}(x)^2q^{(n)}(x,y)\1_{x\neq y},\quad
\bar{\nu}_n=\frac{\nu_n}{\nu_n(\1)}
\end{align*}
as before.
Let $(\xi_s)$
with law $\P^{(n)}_\lambda$ denote the voter model defined by the voting kernel
$q^{(n)}$ with initial distribution $\lambda$. By
convention, $\P^{(n)}_\xi=\P^{(n)}_{\delta_\xi}$ for delta
measures $\delta_\xi$.

In this section, we consider the density processes of these voter models and 
study the necessary and sufficient conditions for the convergence of their second moments to the second moment of the Wright-Fisher diffusion  
$\big(Y,(\P_u)_{u\in [0,1]}\big)$ which is defined by the differential operator $\mathcal G$ in (\ref{eq:FW}). 
Our main result in this section is
Theorem~\ref{thm:2ndmom} below. In the following,
let $\mathbf e$ denote the exponential random
variable with mean $1$, and $\ms L(X)$ denote the law of a random element $X$.

\begin{thm}\label{thm:2ndmom}
Assume that 
\begin{align}\label{eq:Delta}
\lim_{n\to\infty}\piDn=\Delta \in [0,1),
\end{align}
and let $(\gamma_n)$ be a sequence of
constants in $(0,\infty)$. Then the following conditions are equivalent.
\begin{enumerate}
\item [\rm (1)] For some $u\in (0,1)$,
\begin{align}
\lim_{n\to\infty}\E^{(n)}_{\mu_u}[p_1(\xi_{\gamma_n
  t})p_0(\xi_{\gamma_n
  t})]=(1-\Delta)\E_u[Y_t(1-Y_t)]\quad\forall\; t\in\R_+.\label{eq:p1p0WF}
\end{align}
\item [\rm (2)] For all $t\in \R_+$, 
\[
\lim_{n\to\infty}2\gamma_n
\pidn\int_0^t \mathbf P^{(n)}(M_{V,V'}>\gamma_ns)ds
= (1-\Delta)\left(1-e^{-t}\right).
\]
\item [\rm (3)] For all $\mu\in \R_+$, 
\[
\lim_{n\to\infty}2\gamma_n \pidn\mathbf E^{(n)}\left[1-e^{-\mu
      M_{V,V'}/\gamma_n}\right] = 
 (1-\Delta)\E[1-e^{-\mu \mathbf
    e}] .
\]
\item [\rm (4)]
\[
\ms L \left(\dfrac{M_{U,U'}}{\gamma_n}\right) \convdn(1-\Delta)\cdot \ms
L(\mathbf e)+\Delta \cdot\delta_0 .
\]
\end{enumerate}
Moreover, if any of these four conditions holds, then (\ref{eq:p1p0WF}) holds
for any $u\in [0,1]$. 
\end{thm}

\begin{proof}[Proof of Theorem~\ref{thm:2ndmom}] We will
  prove this theorem in the order: 
${\rm (2)}\Longleftrightarrow {\rm (4)}$,
  ${\rm (1)}\Longleftrightarrow {\rm (2)}$, and finally ${\rm (2)}\Longleftrightarrow
  {\rm (3)}$. 

\medskip
\noindent {\bf Step 1:} ${\rm (2)}\Longleftrightarrow {\rm (4)}$. Note
that (4) is equivalent to
\[
\mathbf{P}^{(n)}(M_{U,U'}>\gamma_n t) \lra
(1-\Delta)e^{-t}, \quad \forall\; t>0,
\]
and so it follows immediately from
\eqref{eq:MUV} and (\ref{eq:Delta}) that (2) and (4) are equivalent.

\medskip
\noindent {\bf Step 2:} $(1)\Longleftrightarrow (2)$. Suppose that (1) holds for some
$u\in (0,1)$. Note that
\[
\E_u[Y_t(1-Y_t)]=u(1-u)e^{-t},\quad t\in \R_+.
\]
Using the foregoing equality, \eqref{eq:p1p0mu} and \eqref{eq:MUV}, we see that (1)
implies 
\begin{align*}
(1-\Delta)u(1-u)e^{-t}&=\lim_{n\to\infty}\E^{(n)}_{\mu_u}[p_1(\xi_{\gamma_n
  t})p_0(\xi_{\gamma_nt})]\\ 
&=u(1-u)\left[(1-\Delta) -\lim_{n\to\infty}2\gamma_n\pidn\int_0^t 
\mathbf P^{(n)}(M_{V,V'}>\gamma_ns)ds\right].
\end{align*}
Cancelling out the factor $u(1-u)$ on both sides of the foregoing equality, we obtain (2). For the converse,
we take any $u\in (0,1)$ and then reverse this argument.

\medskip
\noindent {\bf Step 3:} $(2)\Longleftrightarrow (3)$.
Let us make some elementary observations. First, for any
$(0,\infty)$-valued random variable $X$ and any $\mu>0$,  it is elementary to obtain
\begin{align}\label{eq:transid1}
\mathbf{E}\left[1-e^{-\mu X}\right] &=
\mu\int_0^\infty e^{-\mu s} \mathbf{P}(X>s)ds\\
&=
\mu^2\int_0^\infty e^{-\mu t}\int_0^t
\mathbf{P}(X>s)ds dt \label{eq:transid2} .
\end{align}
In addition, \eqref{eq:MUV2} gives
\begin{align}
2\gamma_n \pidn\int_0^\infty \mathbf P^{(n)}(M_{V,V'}>\gamma_n s)ds
\le 1 ,\quad \forall\; n\in\bN.\label{eq:aux8}
\end{align}

Now assume that (2) holds. Taking
$X=M_{V,V'}/\gamma_n$ in \eqref{eq:transid2}, we have for any $\mu>0$
\begin{multline}\label{eq:auxlim}
 2\gamma_n\pidn\mathbf{E}^{(n)}\left[1-e^{-\mu
      M_{V,V'}/\gamma_n}\right]
      =\mu^2
\int_0^\infty e^{-\mu t}\int_0^t
 2\gamma_n\pidn  \mathbf{P}^{(n)}(M_{V,V'}>\gamma_n s)dsdt.
\end{multline}
We pass $n\to\infty$ for both sides of the foregoing equality.
The bound \eqref{eq:aux8} justifying the use of the dominated convergence theorem, the limit of the right-hand side of (\ref{eq:auxlim}) equals
\[
(1-\Delta)\mu^2 \int_0^\infty e^{-\mu t}(1-e^{-t})dt 
= (1-\Delta) \E[1-e^{-\mu\mathbf{e}}],
\]
where the last equality follows from  \eqref{eq:transid2}
with $X=\mathbf{e}$. We have proved (3).

The proof that (3) implies (2) is more involved. Employing
\eqref{eq:transid1} again, we see that
(3) implies that for all $\lambda>0$,
\begin{equation}\label{eq:LAP}
2\gamma_n\pidn\int_0^\infty e^{-\lambda
  s}\mathbf P^{(n)}\left(M_{V,V'}>\gamma_n
  s\right)ds \to (1-\Delta)\int_0^\infty e^{-\lambda s}\P(\mathbf
e>s)ds 
\end{equation}
as $n\to\infty$. 
For any $\mu>0$ and $t>0$, define
\begin{align*}
f_{n,\mu}(t)=&\frac{e^{-\mu t}\mathbf P^{(n)}(M_{V,V'}>\gamma_nt)}{\int_0^\infty e^{-\mu s }\mathbf P^{(n)}(M_{V,V'}>\gamma_ns)ds},\\
f_\mu(t)=&\frac{e^{-\mu t}\P(\mathbf e>t)}{\int_0^\infty e^{-\mu s}\P(\mathbf e>s)ds}.
\end{align*}
Applying \eqref{eq:LAP} twice, we obtain for any $\lambda>0$,
\begin{align*}
\int_0^\infty e^{-\lambda t}f_{n,\mu}(t)dt
&= \dfrac{2\gamma_n\pidn \int_0^\infty e^{-(\lambda+\mu)
  t}\mathbf P^{(n)}\left(M_{V,V'}>\gamma_n
  t\right)dt}
{2\gamma_n\pidn \int_0^\infty e^{-\mu
  s}\mathbf P^{(n)}\left(M_{V,V'}>\gamma_n
  s\right)ds}\\
& \lra \int_0^\infty e^{-\lambda t}f_{\mu}(t)dt 
\end{align*}
as $n\to\infty$. Hence, we deduce from L\'evy's continuity
theorem 
for Laplace transforms of distributions on $\R_+$
(cf. Theorem 4.3 in \cite{Kall}) and (\ref{eq:LAP}) with $\lambda$ replaced by $\mu$ that 
\begin{align}\label{eq:convLap}
&\lim_{n\to\infty}2\gamma_n\pidn\int_0^t e^{-\mu
  s}\mathbf P^{(n)}(M_{V,V'}>\gamma_ns)ds\notag\\ 
&\hspace{1cm}=(1-\Delta)\int_0^t e^{-\mu s}\P(\mathbf
e>s)ds\quad  \forall\;\mu>0,t\geq 0.
\end{align}
Since $\mu>0$ is arbitrary, this gives
\begin{align}
\liminf_{n\to\infty}2\gamma_n \pidn\int_0^t \mathbf
P^{(n)}(M_{V,V'}>\gamma_ns)ds\geq (1-\Delta)\int_0^t
\P(\mathbf e>s)ds.\label{eq:aux6} 
\end{align}

To prove the converse inequality, we start with the decomposition
\begin{align}
\begin{split}
2\gamma_n  \pidn\int_0^t \mathbf P^{(n)}(M_{V,V'}>\gamma_ns)ds=
&2\gamma_n \pidn\int_0^t (1-e^{-\mu s})\mathbf P^{(n)}(M_{V,V'}>\gamma_ns)ds\\
&\quad +2\gamma_n \pidn\int_0^t e^{-\mu s}\mathbf
P^{(n)}(M_{V,V'}>\gamma_ns)ds.
\end{split}
\label{eq:aux9}
\end{align}
Fix any $\mu > 0$. By Markov's inequality and the elementary fact that
$1-e^{-\mu s}\le \mu s$ if $\mu s\ge 0$, the first integral
on the right-hand side above is bounded 
by 
\[
2\gamma_n\pidn \int_0^t \mu 
\dfrac{\mathbf{E}^{(n)}[M_{V,V'}]}{\gamma_n } ds
\le \mu t,
\]
where the last inequality is a consequence of \eqref{eq:MVVmom}.
Applying the foregoing inequality to \eqref{eq:aux9} and using
\eqref{eq:convLap}, we obtain
\[
\limsup_{n\to\infty}2\gamma_n  \pidn\int_0^t \mathbf
P^{(n)}(M_{V,V'}>\gamma_ns)ds \le
\mu t + (1-\Delta)\int_0^t e^{-\mu s}\P(\mathbf e>s)ds .
\]
If we let $\mu\lra 0$ in the above inequality and then
combine the result with \eqref{eq:aux6}, we obtain 
\begin{align}
2\gamma_n\pidn\int_0^t \mathbf P^{(n)}(M_{V,V'}>\gamma_ns)ds\lra (1-\Delta)\int_0^t\P(\mathbf e>s)ds,
\end{align}
which is (2). The proof of the theorem is now complete. 
\end{proof}

\begin{cor}\label{cor:tail}
Under the assumption (\ref{eq:Delta}),
any of (1)--(4) of Theorem~\ref{thm:2ndmom} implies that
\begin{align}\label{eq:exptail}
\lim_{n\to\infty} 2\gamma_n \pidn\mathbf P^{(n)}(M_{V,V'}>\gamma_n t)=(1-\Delta)e^{-t},\quad \forall \;t>0.
\end{align}
If in addition the limit
$m_\Delta = \lim_{n\to\infty}2\gamma_n \pidn$ exists,
then $ m_\Delta\in [1-\Delta,+\infty]$ and 
\begin{align*}
  {\ms L}\left( \frac{M_{V,V'}}{\gamma_n}\right) \convdn
\frac{1-\Delta}{m_\Delta}{\ms L}(\mathbf
  e)+\left(1-\frac{1-\Delta}{m_\Delta}\right)\delta_0
\end{align*}
with the convention that $\frac{1}{+\infty}=0$. 
\end{cor}
\begin{proof}
We prove (\ref{eq:exptail}), from which the second assertion immediately follows. 
We may assume that (2) of Theorem~\ref{thm:2ndmom} holds.
For each $n\geq 1$, define 
\[
f_n(t)= 2\gamma_n \pidn\mathbf P^{(n)}(M_{V,V'}>\gamma_n t),\quad t\in (0,\infty).
\]
Then each $f_n$ is continuous and decreasing. Moreover, by
\eqref{eq:VVtailbound},
\[
0\leq f_n(t)\leq \frac{1}{t}.
\]
Now fix $a>0$ and define
$G_n(t)=1-af_n(t)$,
$t\in[a,\infty)$. By the above inequality, $(G_n)$ is a
sequence of distribution functions on
$[a,\infty)$. Hence by Helly's selection principle, there exist  
a subsequence $(G_{n_k})$ and
some (sub-)distribution function $G$ such that $G_{n_k}(t)\lra G(t)$ for every continuity
point $t\in (a,\infty)$ of $G$. Since $G$ is monotone, it
can have only countably many discontinuity points, and hence
for any 
$a<s<t$,
\[
\lim_{k\to\infty}\int_s^t G_{n_k}(u)du = \int_s^t G(u)du
\]
by dominated convergence.
It then follows from (2) of Theorem~\ref{thm:2ndmom} that
\[
\int_s^t G(u)du = (t-s) + a(1-\Delta)(e^{-t} -e^{-s}),
\]
which implies that $G(t)=1-a(1-\Delta)e^{-t}$ for every
continuity point $t\in (a,\infty)$ of $G$. Since $G$ is increasing,
this equality holds for any $t\in(a,\infty)$. Therefore, $G$ is continuous on $(a,\infty)$ and we have
\[
\lim_{k\to\infty}G_{n_k}(t)= 1-a(1-\Delta)e^{-t} \quad\text{ for any
}t\in(a,\infty) .
\]
As the limit does not depend on the subsequence, this proves that
\[
\lim_{n\to\infty}f_n(t)= (1-\Delta)e^{-t} \quad\text{ for any }t\in(a,\infty).
\]
Since $a>0$ is arbitrary, we have proved \eqref{eq:exptail}.
\end{proof}

We now study what is left out in the conclusion of Corollary~\ref{cor:tail} and
consider, informally, the instant $s_n$ after which the tail
of $2\gamma_n \pidn\mathbf
P^{(n)}(M_{V,V'}/\gamma_n\in \cdot) $ starts to behave like the
$(1-\Delta)$ multiple of the standard exponential distribution. 
The following result will play a crucial role in the proof of
Theorem~\ref{thm:main2}. 

\begin{prop}\label{prop:snextra}
Suppose that (\ref{eq:Delta}) and any of (1)--(4) of Theorem~\ref{thm:2ndmom} holds.  
\begin{enumerate}
\item [\rm (1)] Let $(s_n)\subseteq \R_+$ be any sequence such that \begin{align}\label{eq:defsn-1}
\liminf_{n\to\infty}2\gamma_n\pidn\mathbf P^{(n)}(M_{V,V'}>s_n)\geq  1-\Delta.
\end{align}
Then $s_n=o(\gamma_n)$ as $n\lra\infty$.
\item [\rm (2)] Let $(s_n)\subseteq \R_+$ be any sequence such
  that
  \begin{align}\label{eq:defsn}
\lim_{n\to\infty}2\gamma_n\pidn\mathbf P^{(n)}(M_{V,V'}>s_n)= 1-\Delta.
\end{align}
If $(s'_n)$ is a sequence in $\R_+$ such that $s_n'\geq s_n$ and $s_n'=o(\gamma_n)$, then (\ref{eq:defsn}) holds with $(s'_n)$ in place of $(s_n)$.
\end{enumerate}
\end{prop}
\begin{proof}
Consider (1) first, and we may assume that (4) of Theorem~\ref{thm:2ndmom} holds. Assume the converse that $s_n/\gamma_n$ does not converge to zero.
By passing to a subsequence if necessary, we may assume without loss of
generality that $(s_n)$ satisfies $s_n/\gamma_n\lra
\delta$ for some $\delta\in (0,\infty]$ as $n\lra\infty$. By assumption,
\begin{align*}
1-\Delta\leq &\liminf_{n\to\infty}2\gamma_n\pidn\mathbf P^{(n)}(M_{V,V'}>s_n)\\
\leq &\lim_{n\to\infty}2\gamma_n\pidn\frac{1-\piDn}
{2s_n\pidn}=(1-\Delta)\lim_{n\to\infty}\frac{\gamma_n}{s_n} ,
\end{align*}
where the second inequality is due to \eqref{eq:VVtailbound}.
Hence,
we must have $\delta\leq 1$.
 On the other hand, by
\eqref{eq:MUV} of Corollary~\ref{cor:MUV},
\begin{equation}\label{eq:MUVMUU1} 
\mathbf
P^{(n)}\left(\frac{M_{U,U'}}{\gamma_n}>\frac{s_n}{\gamma_n}\right)
=1-\piDn-2\gamma_n \pidn\mathbf
\int_0^{s_n/\gamma_n} \mathbf{P}^{(n)}(M_{V,V'}>\gamma_ns)ds.
\end{equation}
Using (4) of Theorem~\ref{thm:2ndmom}, we get
\[
\lim_{n\to\infty}\mathbf
P^{(n)}\left(\frac{M_{U,U'}}{\gamma_n}>\frac{s_n}{\gamma_n}\right)=(1-\Delta)e^{-\delta}. 
\]
Apply this to (\ref{eq:MUVMUU1}), and we obtain
\begin{align*}
(1-\Delta)\left(1-e^{-\delta}\right)&=\lim_{n\to\infty}2\gamma_n\pidn\mathbf
\int_0^{s_n/\gamma_n} \mathbf{P}^{(n)}(M_{V,V'}>\gamma_n s)ds\\
&\ge \liminf_{n\to\infty}2\gamma_n\pidn
\left(\dfrac{s_n}{\gamma_n}\right)\mathbf{P}^{(n)}(M_{V,V'}>s_n)\\
&\ge (1-\Delta)\delta
\end{align*}
by the definition of $\delta$ and \eqref{eq:defsn-1}.
As a consequence, $1-e^{-\delta}\ge
\delta$, whereas it is easy to see that
$1-e^{-\delta'}<\delta'$ as long as $\delta'>0$. This proves that $\delta$ must be 0,
so (1) follows.

To prove (2), we let $(s_n')\subset \R_+$ satisfy $s_n'\geq s_n$ and
$s_n'=o(\gamma_n)$. It is immediate that 
\begin{align}
&\limsup_{n\to\infty}2\gamma_n\pidn\mathbf P^{(n)}(M_{V,V'}>s_n')\notag\\
&\hspace{1cm}\leq \lim_{n\to\infty}2\gamma_n\pidn\mathbf
P^{(n)}(M_{V,V'}>s_n)=1-\Delta\label{eq:sn'1}
\end{align}
by the present assumption \eqref{eq:defsn}. 
To obtain the converse inequality, we fix $\vep>0$. 
Since 
$\dfrac{s_n'}{\gamma_n}<\vep$ for all large enough $n$,
\begin{align*}
&2\gamma_n\pidn\mathbf P^{(n)}\left(M_{V,V'}>s_n'\right)\\
&\hspace{1cm}\geq 2\gamma_n\pidn\mathbf
P^{(n)}(M_{V,V'}>\vep \gamma_n)\lra (1-\Delta)e^{-\vep} 
\end{align*}
as $n\lra\infty$ by our assumption on the validity of any of (1)--(4) in Theorem~\ref{thm:2ndmom} and Corollary~\ref{cor:tail}. Since $\vep>0$ is arbitrary, we deduce that
\begin{align}
\liminf_{n\to\infty}2\gamma_n\pidn\mathbf P^{(n)}(M_{V,V'}>s_n')\geq 1-\Delta.\label{eq:sn'2}
\end{align}
We now get the asserted equality \eqref{eq:defsn} for $(s_n')$ from
(\ref{eq:sn'1}) and (\ref{eq:sn'2}). The proof is complete.
\end{proof}

\section{Proof of Theorem~\ref{thm:main}}\label{sec:proof1}
In this section, we prove limit theorems for density
processes. We will focus on the
martingale property of the density processes and use semimartingale limit theorems for our purpose.  
As before, we take
a sequence of irreducible $Q$-matrices $(q^{(n)},E_n)$
with stationary distributions $(\pi^{(n)})$ and
a sequence of strictly positive constants $(\gamma_n)$. 

We first introduce some
notation for density processes used throughout this section.
For each $n$, we write $Y_n=\big(Y_n(t)\big)$ for the density processes
$\big(p_1(\xi_{\gamma_nt})\big)$ of the voter model defined by
$q^{(n)}$. By Proposition~\ref{prop:marts}, each $Y_n$ is a c\`adl\`ag $(\F^n_t)$-martingale, where
\begin{equation}\label{eq:salgebra}
\F_t^n=\sigma(\xi_{\gamma_n s};s\leq t).
\end{equation}
We recall from \eqref{eq:psvp2} that the
predictable quadratic variation process of
$Y_n$ is given by the continuous
process
\begin{equation}\label{eq:psqvp3}
\langle Y_n\rangle_t= \gamma_n\pidn\int_0^t  \left[p_{10}(\xi_{\gamma_ns})+p_{01}(\xi_{\gamma_n s})\right]ds.
\end{equation}
Note that the process in (\ref{eq:psqvp3}) is different from the quadratic variation process, which is given by 
\[
[Y_n]_t=
\sum_{s:s\leq t}\big(\Delta Y_n(s)\big)^2
\]
(see \cite{JS:LT}).

In the following theorem, we refer to \cite{JS:LT} for the notions of $C$-tightness and P-UT condition. 

\begin{thm}\label{thm:lmdp1}
Assume that (\ref{eq:Delta}) holds with $\Delta=0$ and any of (1)--(4) of Theorem~\ref{thm:2ndmom} holds. 
\begin{itemize}
\item [\rm (1)] For any $u\in [0,1]$, the sequence of laws of the c\`adl\`ag martingales
\begin{equation}\label{eq:martseq}
\big(Y_n,\P_{\mu_u}^{(n)}\big), \quad n\in \bN,
\end{equation}
is $C$-tight and this sequence of martingales satisfies the
P-UT condition. 
\item [\rm (2)] For any $u\in [0,1]$, every subsequential limit of the laws of the martingales in (\ref{eq:martseq})
  is the law of a continuous nonnegative martingale bounded by $1$. 
\item [\rm (3)] Suppose that, by choosing a subsequence if necessary, the sequence of laws of the martingales in \eqref{eq:martseq} converges to the law of a continuous
  martingale $Z$. Then a stronger convergence holds:
\begin{align}\label{eq:triconv}
(Y_n,[Y_n],\langle Y_n\rangle)\convdn (Z,[Z],[Z]).
\end{align}
\end{itemize}
\end{thm}
\begin{proof}
  We begin with (1), and we will first show that the sequence of laws of the
  continuous processes
  $\langle Y_n\rangle$ is tight, and in fact $C$-tight. To this end, we apply
 Theorem VI.4.5 of \cite{JS:LT}, so we must verify
  two conditions stated below. First, we check the compact
  containment condition:
\begin{align}\label{ineq:ccc0}
\forall\;\vep>0,\; T>0,\; \exists \; K>0\mbox{ such that
}\sup_{n\in \Bbb N}\Pn_{\mu_u}\left(\sup_{s\leq T}\langle
  Y_n\rangle_s\geq K\right)\leq \vep. 
\end{align}
We make use of the monotonicity of $\langle Y_n\rangle$,
which gives, for every $\vep>0$, $T>0$ and $n\in\bN$,
\begin{align}
  \Pn_{\mu_u}\left(\sup_{s\leq T}\langle Y_n\rangle_s\geq
    K\right)&\leq \frac{\En_{\mu_u}[\langle
    Y_n\rangle_T]}{K}\notag\\ 
 & = \frac{\gamma_n\pidn}{K}\int_0^T\En_{\mu_u}[ 
   p_{10}(\xi_{\gamma_n s})+ p_{01}(\xi_{\gamma_n s})]ds\notag \\
 &\leq \frac{2\gamma_n\pidn}{K}\int_0^T\mathbf
  P^{(n)}(M_{V,V'}>\gamma_ns)ds\label{ineq:ccc0+}\\
&\le \dfrac 1K \label{ineq:ccc},
\end{align}
where (\ref{eq:VVtailbound}) follows from (\ref{eq:p10mu}) and
(\ref{ineq:ccc}) from \eqref{eq:MUV2}. We have proved (\ref{ineq:ccc0}).

The second condition of Theorem~VI.4.5 which we need to check is the Aldous
criterion: 
\begin{align}\label{eq:tightmc}
\forall\; \vep>0, \quad \lim_{\theta \to 0}\limsup_{n\to\infty}\sup_{S,T:S\leq T\leq
  S+\theta}
\Pn_{\mu_u}(|\langle Y_n\rangle_T-\langle Y_n\rangle_S|\geq
\vep)=0,
\end{align}
where $S$ and $T$ range over all finite $(\F_t^n)$-stopping times. 
For any $\theta>0$ and any finite $(\F_t^n)$-stopping times $S$ and $T$ satisfying $S\leq T\leq S+\theta$, we have
\begin{align}
  \En_{\mu_u}[|\langle Y_n\rangle_T-\langle
  Y_n\rangle_S|]&=\gamma_n\pidn
  \En_{\mu_u}\left[\int_S^T\left[p_{10}(\xi_{\gamma_ns})+
      p_{01}(\xi_{\gamma_ns})\right]ds\right]\notag\\ 
& \leq
\gamma_n\pidn\En_{\mu_u}\left[\int_S^{S+\theta}
\left[p_{10}(\xi_{\gamma_ns})+p_{01}(\xi_{\gamma_ns})\right]ds\right]\notag\\ 
  &=\gamma_n\pidn\En_{\mu_u}\left[\E^{(n)}_{\xi_{\gamma_nS}}\left[\int_0^\theta
      \left[p_{10}(\xi_{\gamma_n s})+p_{01}(\xi_{\gamma_n
          s})\right]ds\right]\right]\notag\\
  &\leq 2\gamma_n\pidn \int_0^\theta\mathbf
  P^{(n)}(M_{V,V'}>\gamma_n 
  s)ds,\label{eq:Aldous}
\end{align}
where the last inequality follows from \eqref{eq:p1001bnd}.
Note that the right-hand side of the last
inequality is independent of the stopping times $S$ and
$T$. By assumption, condition (2)  of
Theorem~\ref{thm:2ndmom} holds, 
and thus
\begin{align*}
\lim_{n\to\infty}2\gamma_n\pidn\int_0^\theta\mathbf P^{(n)}(M_{V,V'}>\gamma_n r)dr
=1-e^{-\theta} .
\end{align*}
Our claim (\ref{eq:tightmc}) now follows by applying this equality to the right-hand side of (\ref{eq:Aldous}). 
We have proved that the sequence $(\langle Y_n\rangle)$ of continuous processes
is tight, in fact
  $C$-tight.

 The next step is to prove the desired properties (1)--(3) of the sequence of laws of the c\`adl\`ag martingales $Y_n$, given that we have obtained the $C$-tightness of the sequence of laws of $\langle Y_n\rangle$.
Since the sequence of laws of the initial conditions $\big(Y_n(0)\big)$ is clearly tight, we may
apply Theorem~VI.4.13 of \cite{JS:LT} and conclude that the
sequence $(Y_n,\P_{\mu_u}^{(n)})$ is tight. Since the
jumps of $Y_n$ are uniformly bounded by $\pinmax$,
and $\pinmax\lra 0$ as $n\lra\infty$ on account of the
assumption that $\Delta=0$, it follows from 
Proposition~VI.3.26 of \cite{JS:LT} 
the sequence of laws of $(Y_n)$ is $C$-tight. Finally, by
Proposition~VI.6.13, the P-UT property for $(Y_n)$ holds
too, and so we have prove (1) of our theorem. 

We now consider (2). Suppose that $(Z,\Q_u)$ is a
subsequential limit of the sequence of laws of
$(Y_n,\P_{\mu_u}^{(n)})$. For convenience, we may assume that 
\[
(Y_n,\P_{\mu_u}^{(n)})\convdn (Z,\bQ_u).
\]
Since $(Y_n,\P_{\mu_u}^{(n)})$ is $C$-tight and each member
is a nonnegative martingale uniformly bounded by $1$, it
follows that the limiting object $(Z,\bQ_u)$ is a
continuous martingale bounded by $1$ by Proposition~IX.1.1
in \cite{JS:LT}, and (2) follows. Moreover, the fact that the P-UT property
satisfied by $(Y_n,\P_{\mu_u}^{(n)})$ implies, according to
Corollary~VI.6.30 of \cite{JS:LT}, that
\begin{align}
(Y_n,[Y_n])\convdn (Z,[Z]).\label{eq:Ynconv}
\end{align}

It remains to prove (3), and we need to reinforce the convergence in (\ref{eq:Ynconv})
to (\ref{eq:triconv}).
 To this end, it suffices to show
that the sequence of laws of $\langle Y_n\rangle$ converge
to the law of $[Z]=\langle Z\rangle$  as well. We have shown
the $C$-tightness of the sequence of laws of $\langle
Y_n\rangle$ in the proof of (1). Hence, by
taking a subsequence if necessary, we may assume that the
sequence of laws of $(Y_n,\langle Y_n\rangle)$ converges to
the law of $(Z,B)$ for some continuous increasing process
$B$. The sequence $(Y_n)$ is 
obviously uniformly integrable. We will show in the last paragraph of this proof
that $\{\langle Y_n\rangle_T\}_{n\in \bN}$ is $L_2$-bounded for any $T>0$,
and hence uniformly integrable. It then follows that both $Z$ and
$Z^2-B$ are continuous martingales with respect to the
filtration generated by $Z$ and $B$. The standard
characterization of $\langle Z\rangle$ implies that $\langle
Z\rangle=B$, and we can reinforce the convergence
(\ref{eq:Ynconv}) to 
\begin{align}\label{eq:conv3}
(Y_n,[Y_n],\langle Y_n\rangle)\convdn (Z,[Z],[Z]).
\end{align}

To complete the proof, we verify that for any fixed $T>0$, 
\begin{align}\label{eq:L2PQV}
\sup_{n\in \Bbb N}\E_{\mu_u}^{(n)}[\langle Y_n\rangle_T^2]<\infty. 
\end{align}
With
\eqref{eq:psqvp3} as our starting point, we expand and use
the Markov property at time $s<u$ to obtain
\begin{align*}
\En_{\mu_u}[\langle Y_n\rangle^2_T] &=
2\big(\gamma_n \pidn\big)^2 \int_0^Tds\int_s^T du \;
\En_{\mu_u}\Big[ [p_{10}(\xi_{\gamma_n s})
+p_{01}(\xi_{\gamma_n s})]
\\
&\qquad \times \En_{\xi_{\gamma_n s}}
[p_{10}(\xi_{\gamma_n(u-s)})+p_{01}(\xi_{\gamma_n(u- s)})]
\Big]\\
&\le 
2\big(\gamma_n \pidn\big)^2 \int_0^T
\En_{\mu_u} [p_{10}(\xi_{\gamma_n s})+p_{01}(\xi_{\gamma_n
  s})] ds
\\
&\qquad \times \int_0^T2\mathbf{P}^{(n)}(M_{V,V'}>\gamma_nu)du,
\end{align*}
where the last inequality is due to 
\eqref{eq:p1001bnd}. Applying \eqref{eq:p1001bnd} again,
we obtain
\[
\En_{\mu_u}[\langle Y_n\rangle^2_T] \le
2 \left[2\gamma_n
  \pidn\int_0^T\mathbf{P}^{(n)}(M_{V,V'}>\gamma_nu)du\right]^2
\le 2
\]
by \eqref{eq:MUV2}. This gives (\ref{eq:L2PQV}), and the proof of (\ref{eq:conv3}) is complete.
\end{proof}

\begin{cor}\label{cor:lmdp2}
Suppose that (\ref{eq:diag}) holds (i.e., \eqref{eq:Delta}
holds with $\Delta=0$). Then the convergence \eqref{eq:main1}
implies the mean-field condition \eqref{eq:mfc}.
\end{cor}
  
\begin{proof}
Suppose that the sequence of laws of
$(Y_n,\P^{(n)}_{\mu_u})$ converges to the distribution of
the Wright-Fisher diffusion. This implies that  condition (1) of
Theorem~\ref{thm:2ndmom} holds.
As a consequence, Theorem~\ref{thm:lmdp1}
applies, and thus \eqref{eq:triconv} must hold with the limit
$Z$ distributed as the Wright-Fisher diffusion $Y$ and hence
\begin{align}\label{eq:Zqv}
[Z]_t=\int_0^t Z_s(1-Z_s)ds.
\end{align}
Since
\begin{align}\label{cnt:w(1-w)}
w\lmt \left(\int_0^t w(s)[1-w(s)]ds;t\in \R_+\right): D(\R_+,[0,1])\lra D(\R_+,\R)
\end{align}
defines a continuous function (cf. the proof of Proposition 3.7.1 in \cite{EK:MP}) for any $T\in (0,\infty)$, the equation (\ref{eq:Zqv}) and the convergence (\ref{eq:triconv}) imply
\[
\langle Y_n\rangle_T-\int_0^T Y_n(s)[1-Y_n(s)]ds\convdn 0,
\]
which is exactly the mean-field condition (\ref{eq:mfc}). 
\end{proof}

Our strategy to complete the proof of Theorem~\ref{thm:main}
is to argue that if \eqref{eq:diag} and the mean-field
condition \eqref{eq:mfc} hold then the conditions of
Theorem~\ref{thm:2ndmom} must hold, so that
Theorem~\ref{thm:lmdp1} applies. To do this, we first
show that the mean-field condition is itself a statement of
local convergence in $L^p(\P)$ for any $p\in [1,\infty)$.

\begin{prop}\label{prop:norm}
For any voter model defined by an irreducible $Q$-matrix and
initial configuration $\xi$,
\begin{align}\label{ineq:mombdd}
\E_\xi\left[\left(\int_0^\infty
    \pid[p_{10}(\xi_s)+p_{01}(\xi_s)]ds\right)^m\right]\leq
m!\quad \forall \;  \; m\in \bN.
\end{align}
Hence, the mean-field condition (\ref{eq:mfc}) holds if and
only if for all $T\in (0,\infty)$, 
\begin{multline}\label{eq:wmfnorm}
  \lim_{n\to\infty}\E^{(n)}_{\mu_u}\left[\left|\int_0^T
      \Big(\gamma_n\pidn[p_{10}(\xi_{\gamma_ns})+p_{01}
(\xi_{\gamma_ns})]-p_1(\xi_{\gamma_ns})p_0(\xi_{\gamma_ns})\Big)ds
\right|^p\right]=0\\
 \forall \;p\in
  [1,\infty).
\end{multline}
\end{prop}

If we set $m=2$ and use duality, then the convergence condition (\ref{eq:wmfnorm}) is equivalent to a condition that can be expressed in terms of two pairs of coalescing Markov chains started at different times. We show in the next section that an argument using only a single pair of Markov chains is sufficient to obtain this convergence.

\begin{proof}[Proof of Proposition~\ref{prop:norm}]
By \eqref{eq:p1001bnd} and \eqref{eq:MUV2}, for any initial
configuration $\xi$, 
\begin{equation}\label{eq:m1}
\pid\E_\xi\left[\int_0^\infty [p_{10}(\xi_s)+p_{01}(\xi_s)]ds\right]
\leq 2\pid\int_0^\infty \mathbf P(M_{V,V'}>s)ds 
\le 1 .
\end{equation}
For $m\in \bN$,  if we expand the left-hand side of
\eqref{ineq:mombdd}, and then use the Markov property at
time $s_{m-1}<s_m$, we obtain
\begin{align*}
m!&\big(\pid\big)^m\int_0^\infty ds_1\int_{s_1}^\infty
  ds_2
\cdots\int_{s_{m-1}}^\infty ds_m\;
\E_\xi\left[
\prod_{i=1}^m[p_{10}(\xi_{s_i})+p_{01}(\xi_{s_i})]\right]\\
&=m!\big(\pid\big)^m\int_0^\infty ds_1\int_{s_1}^\infty
  ds_2
\cdots\int_{s_{m-2}}^\infty ds_{m-1}\;
\E_\xi\Bigg[
\prod_{i=1}^{m-1}[p_{10}(\xi_{s_i})+p_{01}(\xi_{s_i})]\\
&\quad \times \E_{\xi_{s_{m-1}}}\int_{s_{m-1}}^\infty
[p_{10}(\xi_{s_m-s_{m-1}})+p_{01}(\xi_{s_m-s_{m-1}})]ds_{m}\Big]\Bigg].
\end{align*}
By applying the bound \eqref{eq:m1} and iteration, we obtain
\eqref{ineq:mombdd}.

For the second assertion, we only need to show that the
mean-field condition implies (\ref{eq:wmfnorm}), because the
converse follows immediately from Markov's
inequality. Moreover, given the mean-field condition, by
Skorokhod's representation and a standard result of uniform
integrability, it is enough to derive a uniform bound on the
$m$-th moment of
\[
\gamma_n\pidn\int_0^\infty [p_{10}(\xi_{\gamma_ns})+p_{01}(\xi_{\gamma_ns})]ds 
\]
for any $m\in \bN$, which is precisely the content of the
first assertion. Hence, (\ref{eq:wmfnorm}) holds, and the
proof is complete. 
\end{proof}

The following result connects the mean-field condition and
the various equivalent conditions in
Theorem~\ref{thm:2ndmom}. It completes the proof of
Theorem~\ref{thm:main}. 

\begin{thm}\label{thm:almostexp}
Suppose that \eqref{eq:diag} holds (i.e., (\ref{eq:Delta})
holds with $\Delta=0$). Then the mean-field condition (\ref{eq:mfc})
implies all of the conditions of Theorem~\ref{thm:2ndmom}
hold, as well as the convergence \eqref{eq:main1} for all
$u\in[0,1]$. 
\end{thm}
\begin{proof}
  Let $(\gamma_n)$ be a sequence of strictly positive constants
  so that the mean-field condition (\ref{eq:mfc}) holds. For
  the first assertion, it is enough to show that
  (1) of Theorem~\ref{thm:2ndmom} holds. By taking a
  subsequence if necessary, we may assume that
\begin{align}
I(t)= \lim_{n\to\infty}\int_0^t
\E^{(n)}_{\mu_u}\left[p_1(\xi_{\gamma_ns})p_0(\xi_{\gamma_ns})\right]ds\label{eq:lmtp0p1} 
\end{align}
exists in $\R_+$ for all $t\in \bQ_+$. Since $s\lmt
\En_{\mu_u}\left[p_1(\xi_{\gamma_ns})p_0(\xi_{\gamma_ns})\right]$
is uniformly bounded, a monotonicity argument implies that
the foregoing limit exists for all $t\in \R_+$ and defines a
continuous function $I$ on $\R_+$. Moreover, given that the
mean-field condition holds, we can  write the
function $I$ as 
\begin{equation}\label{eq:lmtp0p1b}
I(t)=\lim_{n\to\infty}\gamma_n\pidn\int_0^t
\En_{\mu_u}[p_{10}(\xi_{\gamma_n
  s})+p_{01}(\xi_{\gamma_ns})]ds
\end{equation}
by Proposition~\ref{prop:norm}. 
In view of \eqref{eq:time0} and the last display, we obtain
\begin{align}\label{eq:lmtp0p1c}
\lim_{n\to\infty}\E^{(n)}_{\mu_u}[p_1(\xi_{\gamma_n
  t})p_0(\xi_{\gamma_n t})]
=u(1-u) - I(t) \quad\forall \; t\in\R_+.
\end{align}
By the bounded convergence theorem and the definition of
$I(t)$, this implies
\begin{align*}
I(t)=\int_0^t[u(1-u)-I(s)]ds
\end{align*}
by (\ref{eq:lmtp0p1}).
Solving this integral equation gives
$I(t)=u(1-u)(1-e^{-t})$. By plugging this solution into the right-hand side
of \eqref{eq:lmtp0p1c}, we find that
\[
\lim_{n\to\infty}\E^{(n)}_{\mu_u}[p_1(\xi_{\gamma_n
  t})p_0(\xi_{\gamma_n t})] = u(1-u)e^{-t},
\]
which is (1) of Theorem~\ref{thm:2ndmom}. This
proves the first assertion. 

Having proved that the conditions of
Theorem~\ref{thm:2ndmom} hold, we may now apply 
Theorem~\ref{thm:lmdp1}. By (i) of Theorem~\ref{thm:lmdp1}, 
the family $\big(Y_n,\Pn_{\mu_u}\big)$ is $C$-tight for any $u\in [0,1]$.
If $\big(Y_{n_k},\P^{(n_k)}_{\mu_u}\big)$ is any weakly
convergent subsequence, then (ii) and (iii) of Theorem~\ref{thm:lmdp1} imply that 
\[
( Y_{n_k}, \langle Y_{n_k}\rangle) \convdn
( Z, \langle Z\rangle) 
\]
for a continuous martingale $Z$. Thanks to the continuity of the map (\ref{cnt:w(1-w)}), we deduce from the mean-field condition (\ref{eq:mfc}) that
\[
\langle Z\rangle_T=\int_0^T Z_s(1-Z_s)ds\quad\forall \; T\in \R_+ 
\]
almost surely.
Hence, $Z$ is a Wright-Fisher diffusion, and the proof is complete.
\end{proof}

Although we only consider Bernoulli initial conditions throughout this section, the readers may notice that most of the proofs do apply
to the context where for each $n$, the initial condition of the voter model defined by $(q^{(n)},E_n)$ is a general probability measure $\lambda_n$ on $\{0,1\}^{E_n}$. 

More precisely, the same proofs of Theorem~\ref{thm:lmdp1}, Corollary~\ref{cor:lmdp2}, and Proposition~\ref{prop:norm} still apply, if we consider such a generalization. 
For the extension of Theorem~\ref{thm:almostexp}, we consider general initial conditions $\lambda_n$ for which the sequence of laws $\lambda_n(p_1(\xi)\in \cdot)$ converges weakly to a probability measure, say, $\widehat{\lambda}_\infty$ on $[0,1]$, and use (\ref{eq:psvp2}) instead of (\ref{eq:time0}) to obtain an analogue of (\ref{eq:lmtp0p1c}). This leads to the conclusion that, whenever the mean-field condition (\ref{eq:mfc}), with $\mu_u$ replaced by $\lambda_n$ for each voter model defined $q^{(n)}$, holds, we have the weak convergence of the associated density processes to the Wright-Fisher diffusion with initial condition $\widehat{\lambda}_\infty$.

\section{Proof of Theorem~\ref{thm:main2}}\label{sec:proof2}
For the convenience of readers, we give an informal
outline of the proof of Theorem~\ref{thm:main2} first.  We take a
generic voter model as usual and a constant
$\gamma>0$. Falling back in time by a small amount $\delta$
and using the Markov property of voter models, we get for
any instant $s$
\begin{align*}
&\gamma \pid[p_{10}(\xi_{\gamma s})+p_{01}(\xi_{\gamma s})]-p_1(\xi_{\gamma s})p_0(\xi_{\gamma s})\\
&\hspace{1cm}\simeq  \gamma \pid\big(\E_{\xi_{\gamma(s-\delta)}}[p_{10}(\xi_{\gamma\delta})]+\E_{\xi_{\gamma(s-\delta)}}[p_{01}(\xi_{\gamma\delta})]\big)-p_1
\left(\xi_{\gamma(s-\delta)}\right)p_0\left(\xi_{\gamma(s-\delta)}\right)
\end{align*}
on $\sigma(\xi_{\gamma u};u\leq s)$.
We then resort to duality and interpret the right-hand side, or more generally the term
\begin{align}\label{fbmf}
\gamma\pid\big(\E_{\xi}[p_{10}(\xi_{\gamma\delta})]+\E_{\xi}[p_{01}(\xi_{\gamma\delta})]\big)-p_1
\left(\xi\right)p_0\left(\xi\right)
\end{align}
for arbitrary $\xi$, by moving forward in time from the point of view of $q$-Markov chains.
For the first two terms $\E_{\xi}[p_{10}(\xi_{\gamma\delta})]$ and $\E_\xi[p_{01}(\xi_{\gamma\delta})]$, we use Proposition~\ref{prop:pairapprox} and read them
 as expectations of the function 
 \[
 (x,y)\lmt \xi(x)\dxi(y) 
 \]
 of some pairs of $q$-Markov chains before they meet.
On the other hand, $p_1(\xi)$ and $p_0(\xi)$ are the $\pi$-expectations of configurations $\xi$ and $\dxi$, respectively, where $\pi $ is the stationary distribution of the $q$-Markov chain. 
Applying these observations to the quantity (\ref{fbmf}), we can regard (\ref{eq:mfc}) as a result that, informally speaking, the time that a $q$-Markov chain gets close to its equilibrium distribution $\pi$ ``falls far behind" the time that two $q$-Markov chains meet. See also \cite{CMP:CTLVM} for an application of this ``falling-back-moving-forward" argument.

Some additional notation will be useful in the first step of
making the above precise. 
Recall the  system $(X^x_t)$ of independent
$q$-Markov chains on $E$ with semigroup $(q_t)$ and stationary
distribution $\pi$. For any real function $f$ on $E$ define
$\pi(f)=\sum_{x\in E}f(x)\pi(x)$, $q_tf(x) =\sum_{y\in
  E}q_t(x,y)f(y)$, and
\[
{\rm Var}_\pi(f) = \sum_{x\in E}\big(f(x)-\pi(f)\big)^2 \pi(x).
\]

The following bounds will be useful. First,
we have two bounds on the difference between
$q_t\xi(x)$ and $p_1(\xi)$. Recall the
definition of $d_E$ in \eqref{eq:dedef}.
Since $p_1(\xi)=\pi(\xi)$ and $\xi$ is bounded by $1$, it follows that 
\begin{equation}\label{eq:dEbound}
|q_t\xi(x)-p_1(\xi)| \le 2d_E(t), \qquad \forall\;
x\in E, \;\xi\in\{0,1\}^E
\end{equation}
(see Proposition 4.5 in \cite{LPW}).
A second bound (see, e.g., Lemma 2.4 of \cite{DS:LSI}) is
available when $(q_t)$ is reversible and has spectral gap
$\gap$. In this case, for any $f$, 
\begin{equation}\label{eq:gapbound}
{\rm Var}_\pi(q_tf) \le
{\rm Var}_\pi(f) 
e^{-2\gap t}.
\end{equation}
Second, it follows from the definition of $\tmix$ in (\ref{def:tmix}) that
\begin{align}\label{ineq:dEexp}
d_E(k\tmix)\leq e^{-k},\quad \forall\; k\in \Bbb N
\end{align}
(see Section 4.5 of \cite{LPW}).

\begin{prop}\label{prop:main2-1}
Let $(q,E)$ be an irreducible $Q$-matrix. 
For any $0<s<t<\infty$, we have the following estimates. 
\begin{enumerate}
\item [(1)] If $d_E$ denotes the maximal total variation distance defined by (\ref{eq:dedef}), then
\begin{align}
\begin{split}\label{eq:bdd1}
&\sup_{\xi\in \{0,1\}^E}\big|\E_\xi\left[p_{10}(\xi_t)\right]-\mathbf P(M_{V,V'}>s)p_1(\xi)p_0(\xi)\big|\\
&\quad \quad \leq \mathbf P(M_{V,V'}\in (s,t])+4\mathbf P(M_{V,V'}>s)d_E(t-s).
\end{split}
\end{align}
The same inequality holds if $p_{10}$ is replaced by $p_{01}$.
\item [(2)] If the $q$-Markov chain is reversible and $\mathbf g$ is the
 associated spectral gap, then
\begin{align}
\begin{split}\label{eq:bdd2}
&\sup_{\xi\in \{0,1\}^E}\left|\E_\xi[p_{10}(\xi_t)]-\mathbf P(M_{V,V'}>s)p_1(\xi)p_0(\xi)\right|\\
&\quad \quad\leq \mathbf P(M_{V,V'}\in (s,t])+ \frac{2\pi_{\max}q_{\max}}{\nu(\1)}e^{-\mathbf g(t-s)},
\end{split}
\end{align}
where $\pi_{\rm max}=\max\{\pi(x);x\in E\}$ and $q_{\max}=\max\{q(x);x\in E\}$.
The same inequality holds if $p_{10}$ is replaced by $p_{01}$.
\end{enumerate}
\end{prop}
\begin{proof}
The proofs of (1) and (2) are based on the preliminary bound 
\begin{multline}\label{eq:main2-1-2}
\left|\E_\xi[p_{10}(\xi_t)] - p_1(\xi)p_0(\xi)
\mathbf{P}(M_{V,V'}>s)\right| \le \mathbf P(M_{V,V'}\in(s,t])\\ + 
\left|\mathbf E\left[q_{t-s}\xi(X^V_s)
q_{t-s}\dxi(X^{V'}_s)-p_1(\xi)p_0(\xi);M_{V,V'}>s\right]\right| .
\end{multline}
To get this bound, we first use Proposition~\ref{prop:pairapprox} and write for any configuration $\xi$,
\begin{align}
\E_\xi[p_{10}(\xi_{t})]
=&\mathbf E\left[\xi(X^V_t)\dxi(X^{V'}_t);M_{V,V'}>t\right]\notag\\
=&\mathbf E\left[\xi(X^{V}_t)\dxi(X^{V'}_t);M_{V,V'}>s\right]+\vep_1(s,t;\xi),\label{eq:main2-1-1}
\end{align}
where
\begin{align}\label{vep1}
|\vep_1(s,t;\xi)|\leq \mathbf P(M_{V,V'}\in (s,t])
\end{align}
uniformly in $\xi$. Applying the Markov property of the
two-dimensional process $(X^{V},X^{V'})$ at time $s$, we
have 
\begin{align*}
  \mathbf
  E\Big[\xi(X^V_t)\dxi(X^{V'}_t)&;M_{V,V'}>s\Big]=\mathbf
  E\left[\mathbf E\left[\xi(X^v_{t-s})\dxi(X^{v'}_{t-s})\right]
    \Big|_{(v,v')=(X^V_s,X^{V'}_s)};M_{V,V'}>s\right]\\
  &=\mathbf
  E\left[q_{t-s}\xi(X^V_s)q_{t-s}\dxi(X^{V'}_s);M_{V,V'}>s\right]\\
  &=  p_1(\xi)p_0(\xi)\mathbf P(M_{V,V'}>s)\\
  &\qquad+ \mathbf E\left[q_{t-s}\xi(X^V_s)
    q_{t-s}\dxi(X^{V'}_s)-p_1(\xi)p_0(\xi);M_{V,V'}>s\right]
  .
\end{align*}
Combining this equality and the bound (\ref{vep1}) on $\vep_1$ with
\eqref{eq:main2-1-1} gives (\ref{eq:main2-1-2}).

We now consider the proof of (1). The last term in
\eqref{eq:main2-1-2} is bounded above by 
 \begin{align*}
\mathbf E\Big[\Big|&q_{t-s}\xi(X^V_s)
q_{t-s}\dxi(X^{V'}_s)-p_1(\xi)p_0(\xi)\Big|;M_{V,V'}>s\Big] \\
&\le \mathbf
E\left[p_0(\xi)\left|p_1(\xi)-q_{t-s}\xi(X^V_s)\right|;M_{V,V'}>s\right]
\\
&\qquad + \mathbf E\left[q_{t-s}\xi(X^V_s)\left|
p_0(\xi)-q_{t-s}\dxi(X^{V'}_s)\right|;M_{V,V'}>s\right]\\
&\le 4d_E(t-s) \mathbf P(M_{V,V'}>s),
\end{align*}
where we have used \eqref{eq:dEbound}. 
Plugging this bound into \eqref{eq:main2-1-2} gives 
\eqref{eq:bdd1}.

Next, we turn to the proof of (2). In this case, we
bound the last term in \eqref{eq:main2-1-2} in the following way:
\begin{align}
\mathbf E\Big[\Big|q_{t-s}&\xi(X^V_s)q_{t-s}\dxi(X^{V'}_s)-
\pi(\xi)\pi(\dxi)\Big|\Big]\notag\\
&\le \mathbf E\Big[\Big|q_{t-s}\xi(X^V_s)-\pi(\xi)\Big|\Big]+\mathbf E\Big[
\Big|q_{t-s}\dxi(X^{V'}_s)- \pi(\dxi)\Big|\Big]\notag\\
&\le \mathbf E\left[\Big(q_{t-s}\xi(X^V_s)-\pi(\xi)\Big)^2\right]^{1/2}+\mathbf E\left[\Big(q_{t-s}\dxi(X^{V'}_s)-\pi(\dxi)\Big)^2\right]^{1/2}.
\label{eq:Var}
\end{align}
Recall the distribution of $(V,V')$ in (\ref{eq:distVV}).
For all $x\in E$ and $s\geq 0$, we have
\[
\mathbf{P}(X^V_s=x)=\bpi(\{x\}\times E)=\frac{\pi(x)^2q(x)}{\nu(\1)}\le
\frac{\pi_{\max}q_{\max}}{\nu(\1)}\pi(x).
\]
Since $(X^x)$ is independent of $V$, it follows from the foregoing inequality that
\begin{align}\label{eq:VVVbdd}
\mathbf E\left[\big(q_{t-s}\xi(X^V_s)-\pi(\xi)\big)^2\right]
\le \frac{\pi_{\max}q_{\max}}{\nu(\1)}\,{\rm Var}_\pi(q_{t-s}\xi) \le \frac{\pi_{\max}q_{\max}}{\nu(\1)}e^{-2\gap (t-s)} ,
\end{align}
where we have used \eqref{eq:gapbound} and the fact that
${\rm Var}(f)\le 1$ if $|f|$ is bounded by 1. The same bound
holds if we replace $\xi(X^V_s)$ with $\hxi(X^{V'}_s)$. Indeed, we still have
\[
\mathbf  P(X^{V'}_s=x)\leq \frac{\pi_{\max}q_{\max}}{\nu(\1)}\pi(x)\quad\forall\; x\in E\mbox{ and }s\geq 0,
\]
since for any $x\in E$, reversibility implies
\begin{align*}
\mathbf P(V'=x)=&\frac{\sum_{a:a\neq x}\pi(a)^2q(a,x)}{\nu(\1)}\\
= &\frac{\pi(x)\sum_{a:a\neq x}\pi(a)q(x,a)}{\nu(\1)}\leq \frac{\pi_{\max}q_{\max}}{\nu(\1)}\pi(x).
\end{align*}
Hence by (\ref{eq:VVVbdd}) and its analogue when $V$ is replaced by $V'$, we obtain from \eqref{eq:Var} that
\[
\mathbf E\Big[\Big|q_{t-s}\xi(X^V_s)q_{t-s}\dxi(X^{V'}_s)-
\pi(\xi)\pi(\dxi)\Big|\Big]\le \frac{2\pi_{\max}q_{\max}}{\nu(\1)} e^{-\gap (t-s)}.
\]
Plugging this bound into
\eqref{eq:main2-1-2} completes the proof of \eqref{eq:bdd2}. 
\end{proof}

\begin{lem}\label{lem:main2}
  If $\gamma_n=\tmeetn$, then under either condition of
  Theorem~\ref{thm:main2}, any of the conditions in
  Theorem~\ref{thm:2ndmom} holds with $\Delta=0$. Moreover,
  we can choose $(s_n')$ satisfying (\ref{eq:defsn}) with
  $s_n'=o(\gamma_n)$ such that with $\delta_n=
  s_n'/\gamma_n$,
\begin{align}
\vep_n= \sup_{\xi\in \{0,1\}^{E_n}}\left|\gamma_n\pidn\E_\xi^{(n)}\left[
p_{10}(\xi_{\gamma_n\cdot 2\delta_n})+p_{01}(\xi_{\gamma_n\cdot 2\delta_n})\right]-p_1(\xi)p_0(\xi)\right|\xrightarrow[n\to\infty]{} 0.\label{eq:defvepn}
\end{align}
\end{lem}
\begin{proof} 
Let $\gamma_n=\tmeetn=\mathbf E^{(n)}[M_{U,U'}]$. The
strategy is to first prove that 
(4) of Theorem~\ref{thm:2ndmom}
holds, i.e., 
\begin{align}\label{dis:ae}
\frac{M_{U,U'}}{\gamma_n}\xrightarrow[n\to\infty]{{\rm (d)}} \mathbf e,  
\end{align}
and then use
Proposition~\ref{prop:snextra} and the bounds in
Proposition~\ref{prop:main2-1} to choose a 
sequence $(s_n')$ satisfying \eqref{eq:defvepn}.

Suppose first that (i) of Theorem~\ref{thm:main2} holds, and
consider the product chain comprised of
two independent copies of
$q^{(n)}$-Markov chains. For the product chain started   
at its stationary distribution $\pi\otimes\pi$, the first
hitting time of the diagonal $D_n$ has the same law as the meeting time
$M_{U,U'}$. Letting $\big(\tilde q^{(n)}_t\big)$ denote the product chain
semigroup, we have the obvious inequality
\[
\|\tilde q^{(n)}_t(\cdot) -
  \pi^{(n)}\otimes\pi^{(n)}(\cdot)\|_{\rm TV} \le 2 d_{E_n}(t).
\]
By this inequality and our assumption that $\tmixn/\tmeetn\lra 0$, 
Theorem~1.4 of \cite{A:MCEHT} applies to the product chain and gives
\eqref{dis:ae}. 

Now let $(s_n)$ be a sequence with $s_n=o(\gamma_n)$ 
satisfying (\ref{eq:defsn}). Note that the existence of $(s_n)$ follows from the assumption that $\Delta=0$ and the fact that $\mathbf P^{(n)}(V=V')=0$.
Define $(s_n')$ by
\[
s_n'=s_n\vee u_n\mathbf t_{\rm mix}^{(n)},
\] 
where $(u_n)$ satisfies 
\[
\lim_{n\to\infty}u_n=\infty\quad\mbox{and}\quad
\lim_{n\to\infty}\frac{u_n\mathbf t_{\rm mix}^{(n)}}{\gamma_n}= 0.
\] 
Observe that $\delta_n=s_n'/\gamma_n\to 0$ as $n\to\infty$,
and also that  $s_n'/\tmixn\to \infty$ implies $d_{E_n}(s_n')\to
0$ by (\ref{ineq:dEexp}). Furthermore, applying (2) of Proposition~\ref{prop:snextra} to both
$(s_n')$ and $(2s_n')$, we have
\begin{equation}\label{eq:snbounds}
\lim_{n\to\infty}2\gamma_n\pidn\mathbf
P^{(n)}(M_{V,V'}>s'_n)= 1
\;\text{ and }\;
\lim_{n\to\infty}2\gamma_n\pidn\mathbf
P^{(n)}(M_{V,V'}\in(s_n',2s_n'])= 0.
\end{equation}
By (1) of Proposition~\ref{prop:main2-1}, taking
$s=s_n'$ and 
$t=2s_n'$, we have for any initial configuration $\xi$, 
\begin{align}\notag
  \Big|\gamma_n\pidn&\E_\xi^{(n)}\left[p_{10}(\xi_{\gamma_n\cdot
      2\delta_n})+p_{01}(\xi_{\gamma_n\cdot
      2\delta_n})\right] -
  p_1(\xi)p_0(\xi)\Big|\\\notag
&\le 2\gamma_n\pidn\mathbf P^{(n)}(M_{V,V'}\in
  (s_n',2s_n']) +8\gamma_n\pidn\mathbf
  P^{(n)}(M_{V,V'}>s_n')d_{E_n}(s_n')\\
&\qquad +\left|2\gamma_n\pidn\mathbf P^{(n)}(M_{V,V'}> s'_n)-1\right|p_1(\xi)p_0(\xi)
.\label{eq:main2-1}
\end{align}
Therefore, \eqref{eq:defvepn} follows from
\eqref{eq:snbounds} and \eqref{eq:main2-1}.

Next, suppose that (ii) of
Theorem~\ref{thm:main2} holds, so
$\gap_n \tmeetn\to\infty$ as $n\to\infty$. 
We consider again the 
product chain, the hitting time of the diagonal and the
meeting time $M_{U,U'}$. 
The product chain is reversible, and has spectral 
gap $\tilde \gap_n=\gap_n/2$ by Lemma~3.2 in
\cite{DS:LSI}. It follows from Proposition
3.23 in \cite{AF:MC} that the hitting time for the diagonal ${\tt D}_n$
is approximately exponentially distributed in the sense that
\eqref{dis:ae} holds. 

We again select a sequence
$(s_n)$ such that $s_n=o(\gamma_n)$ and
\eqref{eq:defsn} holds. The existence of $(s_n)$ is due to the same reason as in the case (i).
Now we choose $(u_n)$ 
such that 
\begin{equation}\label{eq:un2}
\lim_{n\to\infty}u_n=\infty\quad \text{ and }\quad \lim_{n\to\infty}u_n\dfrac{\log(e\vee
  \gamma_n\pinmax q^{(n)}_{\max} )}{\gap_n\gamma_n}= 0,
\end{equation}
and define $(s_n')$ by
\[
s'_n = s_n \vee \dfrac{\gamma_n}{u_n} .
\]
Clearly $\delta_n=s_n'/\gamma_n\to 0$, and
\eqref{eq:snbounds} holds by (2) of Proposition~\ref{prop:snextra}. By (2)
of Proposition~\ref{prop:main2-1} with $s=s_n'$ and
$t=2s_n'$, we get for any initial configuration $\xi$,
\begin{align}\notag
  \Big|\gamma_n\pidn&\E_\xi^{(n)}\left[p_{10}(\xi_{\gamma_n\cdot
      2\delta_n})+p_{01}(\xi_{\gamma_n\cdot
      2\delta_n})\right] -
  p_1(\xi)p_0(\xi)\Big|\\\notag
&\le
2\gamma_n\pidn\mathbf P^{(n)}(M_{V,V'}\in
  (s_n',2s_n'])+\frac{4\pi^{(n)}_{\max}q^{(n)}_{\max}}{\nu_n(\1)}\gamma_n\pidn e^{-\gap_n s'_n}\\
  &\qquad + \left|2\gamma_n\pidn\mathbf P^{(n)}(M_{V,V'}>s'_n)-1\right|p_1(\xi)p_0(\xi). \label{eq:main2-3}
\end{align}
As before, by our choice of $(s_n')$ and 
Proposition~\ref{prop:snextra}, the first term and the third one on the right-hand
side above tend to 0 as $n\to\infty$. 

To show that the second term on the right-hand side of (\ref{eq:main2-3}) also tends to zero, we make some observations for the condition (ii) of 
Theorem~\ref{thm:main2}. Now, $\piDn\to 0$, and so the inequality (\ref{eq:tmeetnlbd}) implies that 
\[
\liminf_{n\to\infty}\gamma_n \pinmax q^{(n)}_{\max}=\liminf_{n\to\infty}\tmeetn \pinmax q^{(n)}_{\max}>0.
\]
On the other hand,
\begin{align*}
\frac{\gap_n\gamma_n}{u_n}-\log \left(e\vee \gamma_n\pinmax q^{(n)}_{\max}\right)=&\log\left (e\vee \gamma_n \pinmax q^{(n)}_{\max}\right)\left(\frac{\gap_n\gamma_n}{u_n\log\big (e\vee \gamma_n\pinmax q^{(n)}_{\max}\big)}-1\right)\\
\geq &\frac{\gap_n\gamma_n}{u_n\log \big(e\vee \gamma_n \pinmax q^{(n)}_{\max}\big)}-1\lra \infty,
\end{align*}
where the convergence follows from the choice of $u_n$ in (\ref{eq:un2}). We deduce from 
the last two displays that
\[
\gap_ns_n'-\log \left(e\vee \gamma_n \pinmax q^{(n)}_{\max}\right)\geq \frac{\gap_n\gamma_n}{u_n}-\log \left(e\vee \gamma_n \pinmax q^{(n)}_{\max} \right)\lra \infty,
\]
which is enough for the desired convergence. The proof is complete.
\end{proof}

We are now ready to prove Theorem~\ref{thm:main2}. 

\begin{proof}[\bf Proof of Theorem~\ref{thm:main2}]
  We have shown in the proof of Lemma~\ref{lem:main2} that all
  of the equivalent conditions of Theorem~\ref{thm:2ndmom}
  hold.  Also, the sequences $(\delta_n)$ and
  $(\vep_n)$ defined in Lemma~\ref{lem:main2} satisfy $\delta_n\lra 0$
  and $\vep_n\lra 0$ as $n\lra\infty$. 
  
Our goal in this proof is to prove the $L^1$-norm version of the mean-field condition, namely (\ref{eq:wmfnorm}) with $p=1$ for any $T>0$. For this, we first note that  \eqref{eq:p1001bnd} gives
\begin{align*}
\E^{(n)}_{\mu_u}\left[\int_0^{2\delta_n}\Big|\gamma_n\pidn[p_{10}(\xi_{\gamma_n
      s})+p_{01}(\xi_{\gamma_n s})]-p_1(\xi_{\gamma_n
      s})p_0(\xi_{\gamma_ns})\Big|ds\right]\\
\le 2\gamma_n\pidn\int_0^{2\delta_n}
\mathbf{P}^{(n)}(M_{V,V'}>\gamma_ns)ds + 2\delta_n,
\end{align*}
and the right-hand side tends to 0
as $n\lra\infty$ by (2) of
Theorem~\ref{thm:2ndmom} and the fact that $\delta_n\lra 0$. Hence, it remains to show that
\begin{equation}\label{eq:mfcgoal}
\lim_{n\to\infty}\E^{(n)}_{\mu_u}\left[\left|\int_{2\delta_n}^T\left(\gamma_n\pidn[p_{10}(\xi_{\gamma_n s})+p_{01}(\xi_{\gamma_n s})]-p_1(\xi_{\gamma_n s})p_0(\xi_{\gamma_ns})\right)ds\right|\right]=0,
\end{equation}
for any $T>0$.

For convenience, we write from now on
\[
\bar p(\xi) \equiv 
p_{10}(\xi) + p_{01}(\xi),
\]
and for any
$s\ge 2\delta_n$,
\begin{align*}
H_n(s)\equiv \gamma_n\pidn \bar p(\xi_{\gamma_n s})
-\En_{\mu_u}\left[\gamma_n\pidn \bar p(\xi_{\gamma_n s})
\big|\F^n_{s-2\delta_n}\right]
\end{align*}
(recall the definition of $\F^n_{t}$ from \eqref{eq:salgebra}). Note that 
$H_n(s)\in \F^n_{s-2\delta_n}$.
Then
\begin{align}
\E^{(n)}_{\mu_u}&\left[\left|\int_{2\delta_n}^T\gamma_n\pidn \bar p(\xi_{\gamma_n s})-p_1(\xi_{\gamma_n s})p_0(\xi_{\gamma_ns})ds\right|\right]\notag\\
\begin{split}\label{eq:aux11}
  &\leq \E^{(n)}_{\mu_u}\Bigg[\Bigg(\int_{2\delta_n}^T H_n(s)ds\Bigg)^2\Bigg]^{1/2}\\
  &\quad+\E^{(n)}_{\mu_u}\Bigg[\int_{2\delta_n}^T\Big|\En_{\mu_u}\left[\gamma_n\pidn \bar p(\xi_{\gamma_n s})
\big|\F_{s-2\delta_n}^n\right]
  -p_1(\xi_{\gamma_n(s-2\delta_n)})
  p_0(\xi_{\gamma_n(s-2\delta_n)})\Big|ds
  \Bigg]\\
  &\quad+\E^{(n)}_{\mu_u}\left[\left|\int_{2\delta_n}^Tp_1(\xi_{\gamma_n(s-2\delta_n)})
      p_0(\xi_{\gamma_n(s-2\delta_n)})-p_1(\xi_{\gamma_ns})p_0(\xi_{\gamma_n
        s}) ds\right|\right],
\end{split}
\end{align}
and so to verify \eqref{eq:mfcgoal} it suffices to prove that
each term on the right-hand side of the above tends to 0 as $n\lra\infty$.

We first prove that the first term on the right-hand side of (\ref{eq:aux11}) tends to zero.
Note that
\begin{align}\label{eq:main2-2}
\E^{(n)}_{\mu_u}\left[\left(\int_{2\delta_n}^T H_n(s)ds\right)^2\right]
=2\E^{(n)}_{\mu_u}\left[\int\!\!\!\int_{2\delta_n\leq r\leq s\leq T} H_n(s)H_n(r) \1_{r>s-2\delta_n}\,dsdr\right].
\end{align}
To justify the restriction ``$\1_{r>s-2\delta_n}$'' for the right-hand side, we note that
for $2\delta_n\leq r<s-2\delta_n$, 
\[
\E^{(n)}_{\mu_u }\left[H_n(s)|\F^n_r\right]=0, 
\]
and hence, we obtain by conditioning on $\F_r^n$ that 
\[
\E^{(n)}_{\mu_u}[H_n(s)H_n(r)]=0,\quad 2\delta_n\leq r<s-2\delta_n.
\]
Now expanding 
$H_n(r)H_n(s)$, we obtain 
\begin{align}
\E^{(n)}_{\mu_u}&\left[\int\!\!\!\int_{2\delta_n\leq r\leq s\leq T} \,drds \,H_n(s)H_n(r) \1_{r>s-2\delta_n}\right]\notag\\
\begin{split}\label{eq:expand}
  &=\E_{\mu_u}^{(n)}\left[\int_{2\delta_n}^Tdr\int_r^{T\wedge
      (r+2\delta_n)}ds\,\big(\gamma_n\pidn\big)^2 \bar
    p(\xi_{\gamma_nr})\bar p(\xi_{\gamma_n s})\right]\\
  &\quad-\int_{2\delta_n}^Tdr\int_r^{T\wedge
    (r+2\delta_n)}ds\,\E^{(n)}_{\mu_u}\Big[\big(\gamma_n\pidn\big)^2\En_{\mu_u}
\left[\bar p(\xi_{\gamma_nr})
    |\F^n_{r-2\delta_n}\right]\bar p(\xi_{\gamma_n s})\Big]\\
  &\quad-\int_{2\delta_n}^Tdr\int_r^{T\wedge(r+2\delta_n)}ds
\, \E^{(n)}_{\mu_u}\Big[\big(\gamma_n\pidn\big)^2   
  \bar p(\xi_{\gamma_n r})
 \En_{\mu_u}\left[\bar p(\xi_{\gamma_n
      s})\big|\F^n_{s-2\delta_n}\right]\Big]\\ 
  &\quad +\int_{2\delta_n}^Tdr\int_r^{T\wedge (r+2\delta_n)}ds
  \, \E^{(n)}_{\mu_u}\Big[\big(\gamma_n\pidn\big)^2\En_{\mu_u}\left[\bar
    p(\xi_{\gamma_n r})\big|\F^n_{r-2\delta_n}\right] \\
 &\qquad\qquad\times\En_{\mu_u}\big[\bar p(\xi_{\gamma_n
     s})\big|\F^n_{s-2\delta_n}\big]\Big]. 
\end{split}
\end{align}

We will show that each of the four terms on the right-hand
side of the last equality tends to zero as $n\to\infty$.
To do this we first state three facts which we will use
repeatedly. By our choice of
$s'_n$ and $\delta_n=s_n'/\gamma_n$ in
Lemma~\ref{lem:main2}, and by
Proposition~\ref{prop:snextra}, 
\begin{equation}\label{eq:gammandeltan}
\lim_{n\to\infty} 2\gamma_n\pidn \mathbf{P}^{(n)}(M_{V,V'}>
2\gamma_n \delta_n ) = 1.
\end{equation}
By (2) of Theorem~\ref{thm:2ndmom}, for each $t>0$,
\begin{equation}\label{eq:kt}
K_t = \sup_{n\in \Bbb N} 2\gamma_n\pidn\int_0^t
\mathbf P^{(n)}(M_{V,V'}>\gamma_n s)ds <\infty .
\end{equation}
Finally, by Markov property and \eqref{eq:p1001bnd}, we have for
$r<s$, 
\begin{equation}\label{eq:goodbnd}
\En_{\mu_u} \left[ \bar p(\xi_{\gamma_n s}) \big|\F^n_{r} \right]=
\En_{\xi_{\gamma_n r}}\left[ \bar p(\xi_{\gamma_n (s-r)})
\right]\le 2\mathbf{P}^{(n)}\big(M_{V,V'}>\gamma_n (s-r) \big).
\end{equation}

\medskip We start with the first term on the right-hand side of
(\ref{eq:expand}), arguing in more detail than we will for the other terms. By
conditioning at time $r<s$ and using \eqref{eq:goodbnd}
repeatedly, we obtain
\begin{align*}
  \E^{(n)}_{\mu_u}&\left[\int_{2\delta_n}^Tdr\int_r^{T\wedge
      (r+2\delta_n)}ds\,\big(\gamma_n\pidn\big)^2\,
    \bar p(\xi_{\gamma_nr})\,
    \bar p(\xi_{\gamma_n s})\right]\\ 
  &=\int_{2\delta_n}^Tdr\int_r^{T\wedge
      (r+2\delta_n)}ds\,
  \E^{(n)}_{\mu_u}\Big[\gamma_n\pidn
  \bar p(\xi_{\gamma_nr})
  \En_{\mu_u}   \Big[\gamma_n\pidn \bar
  p(\xi_{\gamma_ns})\big|\F^n_{r} \Big]\Big]\\  
  &\leq\int_{2\delta_n}^Tdr\int_r^{T\wedge
      (r+2\delta_n)}ds\,\E^{(n)}_{\mu_u}\Big[\gamma_n\pidn 
  \bar p(\xi_{\gamma_nr})
   \Big]
  2\gamma_n\pidn \mathbf{P}^{(n)}\big(M_{V,V'}>\gamma_n(s-r)\big)
\\  
  &\leq2\gamma_n\pidn \int_{2\delta_n}^T
    \mathbf{P}^{(n)}(M_{V,V'}>\gamma_nr )dr\times
2\gamma_n\pidn  \int_0^{2\delta_n} \mathbf{P}^{(n)}(
M_{V,V'}>\gamma_ns) ds
  \\ 
&\le K_T \times 2\gamma_n\pidn\int_0^{2\delta_n} \mathbf{P}^{(n)}(
M_{V,V'}>\gamma_ns)\,ds
  \lra 0 \quad
  \text{ as }n\lra\infty,
\end{align*}
where we have used \eqref{eq:p1001bnd},
\eqref{eq:kt} and (2) of Theorem~\ref{thm:2ndmom}. 

\medskip
For the second term on the right side of (\ref{eq:expand}),
again applying \eqref{eq:goodbnd}
repeatedly, we obtain
\begin{align*}
0&\leq\int_{2\delta_n}^Tdr\int_r^{T\wedge
  (r+2\delta_n)}ds\, \En_{\mu_u}\left[\big(\gamma_n\pidn\big)^2
\En_{\mu_u}\left[\bar p(\xi_{\gamma_nr})
\big|\F^n_{r-2\delta_n}\right]
\bar p(\xi_{\gamma_n s})\right]\\
&\leq\int_{2\delta_n}^Tdr\int_r^{T\wedge (r+ 2\delta_n)} ds
\,2\gamma_n\pidn
\mathbf{P}^{(n)}(M_{V,V'}>\gamma_n2\delta_n)
2\gamma_n\pidn \mathbf{P}^{(n)}(M_{V,V'}>\gamma_ns)\\
&\leq\int_{2\delta_n}^Tdr\int_r^{T\wedge (r+ 2\delta_n)} ds
\, [2\gamma_n\pidn
\mathbf{P}^{(n)}(M_{V,V'}>2\gamma_n\delta_n)]^2\\
&\leq 2\delta_nT \times \big(2\gamma_n\pidn
\mathbf{P}^{(n)}(M_{V,V'}>2\gamma_n\delta_n)\big)^2
\lra 0 \quad\text{ as }n\lra\infty,
\end{align*}
where we have made use of the fact that $s\ge r\ge
2\delta_n$ above, \eqref{eq:gammandeltan} 
and the fact that $\delta_n\to 0$. 

The third term on the right-hand side of
(\ref{eq:expand}) is slightly different from the previous one. Now, we use \eqref{eq:goodbnd} in the following way:
\begin{align*}
0&\leq\int_{2\delta_n}^Tdr\int_r^{T\wedge
  (r+2\delta_n)}ds\, \En_{\mu_u}\left[\big(\gamma_n\pidn\big)^2
\En_{\mu_u}\left[\bar p(\xi_{\gamma_ns})
\big|\F^n_{s-2\delta_n}\right]
\bar p(\xi_{\gamma_n r})\right]\\
&\leq\int_{2\delta_n}^Tdr\int_r^{T\wedge (r+ 2\delta_n)} ds
\, 2\gamma_n\pidn
\mathbf{P}^{(n)}(M_{V,V'}>2\gamma_n\delta_n)
2\gamma_n\pidn \mathbf{P}^{(n)}(M_{V,V'}>\gamma_nr)\\
&\leq 2\delta_n\times
2\gamma_n\pidn
\mathbf{P}^{(n)}(M_{V,V'}>\gamma_n2\delta_n)
\int_{0}^Tdr
2\gamma_n\pidn \mathbf{P}^{(n)}(M_{V,V'}>\gamma_nr)
\\
& \le 2\delta_n\times 2\gamma_n\pidn
\mathbf{P}^{(n)}(M_{V,V'}>\gamma_n2\delta_n)\times K_T
\lra 0\quad \text{ as }n\lra\infty,
\end{align*}
which follows from (\ref{eq:gammandeltan}), (\ref{eq:kt}), and the fact that $\delta_n\to 0$. 

\medskip
Finally, for the last term on the right-hand side of (\ref{eq:expand}), the bound \eqref{eq:goodbnd} remains useful and we get
\begin{align*}
0\leq &\int_{2\delta_n}^Tdr\int_r^{T\wedge (r+2\delta_n)}ds
\, \En_{\mu_u}\Big[\En_{\mu_u}\left[\gamma_n\pidn\bar p(\xi_{\gamma_n
    r})\big|\F^n_{r-2\delta_n}\right]\\ 
&\qquad \times\En_{\mu_u}\left[\gamma_n\pidn\bar
  p(\xi_{\gamma_n s}) |\F^n_{s-2\delta_n}\right]\Big]\\
&\le \int_{2\delta_n}^Tdr\int_r^{T\wedge (r+2\delta_n)}ds
\, \Big(2\gamma_n\pidn \mathbf{P}^{(n)}(M_{V,V'}>2\gamma_n\delta_n)\Big)^2\\
&\leq 2\delta_nT\times\Big(2\gamma_n\pidn\mathbf
P^{(n)}(M_{V,V'}>2\gamma_n\delta_n)\Big)^2\lra 0\quad \text{
  as }n\lra\infty
\end{align*}
since $\delta_n\to 0$ and we have (\ref{eq:gammandeltan}).
We have thus verified the desired convergence for the first term of (\ref{eq:aux11}).

We now make some observations for the other two terms in (\ref{eq:aux11}).
To
handle the second term, we apply the Markov property of the
voter model to the integrand at time $s-2\delta_n$. It
follows from Lemma~\ref{lem:main2} that the integrand 
\[
\Big|\En_{\mu_u}\big[\gamma_n\pidn \bar p(\xi_{\gamma_n s})
\big|\F_{s-2\delta_n}^n\big]
  -p_1\big(\xi_{\gamma_n(s-2\delta_n)}\big)
  p_0\big(\xi_{\gamma_n(s-2\delta_n)}\big)\Big|
\]
is
uniformly bounded by $\vep_n$, so that the second term is no
larger than $\vep_nT$. By a simple change-of-variable argument, the third term
above is easily seen to be bounded by $4\delta_n$.  
Since both of the sequences
$(\vep_n)$ and $(\delta_n)$ tend to zero, the last two terms in (\ref{eq:aux11}) both tend to zero.
This completes
the proof of Theorem~\ref{thm:main2}. 
\end{proof}

\section{Coalescence times and density processes}\label{sec:cmc}
Let $(D_t)$ be the pure-death process on $\bN$ which jumps from
  $k$ to $k-1$ at rate $\binom k2$, $k\ge 2$. Set $Z_1=\infty$, and recall that we let $Z_2,Z_3,\cdots$ be independent exponential variables with mean $\mathbf E[Z_j]=1/{j\choose 2}$. For any integer $k\geq 2$, it is easy to see from independence of $Z_{j}$ and $\sum_{i=j+1}^kZ_i$ that
\[
\mathbf{P}_k(D_t=j) = \mathbf P\left(\sum_{i=j+1}^k Z_i\leq  t <
  \sum_{i=j}^k Z_i
\right), \qquad 1\le j\le k
\]
Furthermore, $(D_t)$ and the Wright-Fisher diffusion $(Y_t)$ are linked by the following
duality equation (see Equation (7.21) of \cite{Ta}):
\begin{equation}\label{eq:dualityWF}
\E_u[ Y^k_t] = \mathbf{E}_k[u^{D_t}],
\quad \forall\; u\in[0,1],\; k\in\bN,\;t\in \R_+. 
\end{equation}
The proofs of Proposition~\ref{prop:ccmc} and Proposition~\ref{prop:ccmc2} are both based on this simple equality.

\begin{proof}[\bf Proof of Proposition~\ref{prop:ccmc}]
 Let us fix $t>0$ and $k\ge 2$. By the duality equation
  \eqref{eq:duality},  and the
  fact that the initial law of $\xi_0$ is $\mu_u$, 
\begin{align}
\E_{\mu_u}^{(n)}\Big[\big(p_1(\xi_{\gamma_n t})\big)^k\Big] = &
\sum_{j=1}^ku^j\mathbf{P}^{(n)}\left(\left|\left\{\hat X^{U_1}_{t\gamma_n},\dots,
\hat X^{U_k}_{t\gamma_n}\right\}\right| =j\right) 
\notag\\
=&\sum_{j=1}^ku^j\mathbf P^{(n)}\left(\cnkj \leq \gamma_n t<
  {\sf C}^{(n)}_{k,j-1}\right) , \label{eq:ccnv}
\end{align}
with the convention that ${\sf C}^{(n)}_{k,k}=0$. 
On the other hand, by assumption and the 
duality equation \eqref{eq:dualityWF},
\begin{align}
\lim_{n\to\infty}\E_{\mu_u}^{(n)}\left[\big(p_1(\xi_{\gamma_n t})\big)^k\right]=& \E_u[Y^k_t]\notag \\
=&\mathbf E_k[u^{D_t}]
=\sum_{j=1}^k u^j \mathbf P_k[D_t=j]\notag\\
=&\sum_{j=1}^k u^j \mathbf P\left(\sum_{i=j+1}^k Z_i\leq t<\sum_{i=j}^k
Z_i\right).\label{eq:ccnv2}
\end{align}
Combining \eqref{eq:ccnv} and
\eqref{eq:ccnv2} we see that 
\[
\lim_{n\to\infty}\sum_{j=1}^ku^j\mathbf P^{(n)}\left(\cnkj \leq \gamma_n t<
  {{\sf C}^{(n)}_{k,j-1}}\right) =
 \sum_{j=1}^k u^j \mathbf P\left(\sum_{i=j+1}^k Z_i\leq t<\sum_{i=j}^k
Z_i\right).
\]
The foregoing equality holds for all $u\in[0,1]$, so
it must be the case that
\[
\lim_{n\to\infty}\mathbf{P}^{(n)}\left(\cnkj \leq \gamma_n t<{{\sf
    C}^{(n)}_{k,j-1}}\right) =
\mathbf{P}\left (\sum_{i=j+1}^k
  Z_i\leq t<\sum_{i=j}^kZ_i\right),\quad \forall\; 1\leq j\leq k.
\]
It follows by dominated convergence that for any $\lambda> 0$
\[
\lim_{n\to\infty}\int_0^\infty \lambda e^{-\lambda t}\mathbf{P}^{(n)}\left(\cnkj \leq \gamma_n t<{{\sf
    C}^{(n)}_{k,j-1}}\right)dt =
\int_0^\infty \lambda e^{-\lambda t}\mathbf{P}\left (\sum_{i=j+1}^k
  Z_i\leq t<\sum_{i=j}^kZ_i\right)dt
\]
and hence 
\[
\lim_{n\to\infty}\left(\mathbf{E}^{(n)}\big[e^{-\lambda {\sf
    C}^{(n)}_{k,j-1}/\gamma_n}\big]-\mathbf E^{(n)}\big[e^{-\lambda{\sf
    C}^{(n)}_{k,j}/\gamma_n }\big]\right)=\mathbf E\big[e^{-\lambda \sum_{i=j}^kZ_i}\big]-\mathbf E\big[e^{-\lambda \sum_{i=j+1}^kZ_i}\big]
\]
for any $1\leq j\leq k$
and our assertion follows plainly. 
\end{proof}

\begin{proof}[\bf Proof of Proposition~\ref{prop:ccmc2}]
The proof of Proposition~\ref{prop:ccmc2} is a slight generalization of Proposition~\ref{prop:ccmc}, so we will skip some details. We start with two equalities.
First, as in (\ref{eq:ccnv}), we have
\begin{align}\label{ineq:tauWF0}
\begin{split}
\P_{\mu_u}^{(n)}\left(\tau^{(n)}_1\leq \gamma_nt\right)=&\E_{\mu_u}^{(n)}\left[\prod_{x\in E_n}\xi_{\gamma_nt}(x)\right]\\
=&\sum_{j=1}^{|E_n|}u^j\mathbf P^{(n)}\left(\hat{\sf C}^{(n)}_{j} \leq \gamma_n t<
 \hat{\sf C}^{(n)}_{j-1}\right).
 \end{split}
\end{align}
Also by (\ref{eq:dualityWF}), we have
\begin{align}\label{ineq:tauWF2}
\P_u\left(\tau_1^{Y}\leq t\right)=\lim_{k\to\infty}\E_u[Y_t^k]=\sum_{j=1}^\infty u^j
\mathbf P\left(\sum_{i=j+1}^\infty Z_i\leq t<\sum_{i=j}^\infty Z_i\right).
\end{align}
That (\ref{eq:full2}) implies (\ref{eq:full1}) now follows from the two displays (\ref{ineq:tauWF0}) and (\ref{ineq:tauWF2}) and dominated convergence.

The converse also uses the same two displays, but now we need another elementary result: For any nonnegative $a^n_j$, for $n,j\in \Bbb N$, with $\sum_j a^n_j\leq 1$, the condition that
\begin{align}\label{ineq:tauWF3}
\lim_{n\to\infty}\sum_j a^n_ju^j\quad \mbox{exists for every $u\in (0,1)$}
\end{align}
is enough to obtain that $\lim_{n\to\infty}a^n_j$ exists for every $j\in \Bbb N$. Indeed, if $(n_k)$ and $(n'_k)$ are two subsequences such that $\lim_{k\to\infty}a^{n_k}_j$ and $\lim_{k\to\infty}a^{n'_k}_j$ exist for all $j\in \Bbb N$, then
the limits are all in $[0,1]$,
and so by dominated convergence
(\ref{ineq:tauWF3}) implies
\begin{align*}
\sum_j \left(\lim_{k\to\infty}a^{n_k}_j\right)u^j=\sum_j \left(\lim_{k\to\infty}a^{n'_k}_j\right)u^j,\quad u\in [0,1).
\end{align*}
We deduce from these that 
$\lim_{k\to\infty}a^{n_k}_j=\lim_{k\to\infty}a^{n'_k}_j$ for all $j\in \Bbb N$, which, by a diagonal argument on selecting convergent subsequences of $(a^n_j)_{n\in \Bbb N}$ each $j\in \Bbb N$, is enough for our claim that $\lim_{n\to\infty}a^n_j$ exists for every $j\in \Bbb N$.

Using this elementary result, and assuming (\ref{eq:full1})
so that the right-hand side of (\ref{ineq:tauWF0}) converges to the right-hand side of (\ref{ineq:tauWF2}), we obtain that for every $j\in \Bbb N$, 
\[
\lim_{n\to\infty}\mathbf P^{(n)}\left(\hat{{\sf C}}^{(n)}_{j} \leq \gamma_nt < \hat{\sf C}^{(n)}_{j-1}\right)
\]
must exist and this limit must be
\[
\mathbf P\left(\sum_{i=j+1}^\infty Z_i\leq t<\sum_{i=j}^\infty Z_i\right).
\]
This establishes (\ref{eq:full2}), and the proof is complete.
\end{proof}

\section{Examples}\label{sec:eg}
 In this section, we consider various sequences of $(q^{(n)},E_n)$-Markov chains for which 
one of the conditions of
Theorem~\ref{thm:main2} and Corollary~\ref{cor:main} applies, and hence the convergence of
the corresponding voter model densities in \eqref{eq:main1} holds.

The $Q$-matrices $q^{(n)}$ considered below are of the form $q^{(n)}=p^{(n)}-{\rm Id}_{E_n}$, where $p^{(n)}$ is a symmetric probability matrix but not necessarily has zero diagonal. 
In this case, $1-\lambda$ is an eigenvalue of $-q^{(n)}$ if and only if $\lambda$ is an eigenvalue of $p^{(n)}$. If in addition $p^{(n)}$ has zero diagonal, the
inequality \eqref{eq:meetlowerbound} for such a particular $Q$-matrix $q^{(n)}$ becomes
\begin{equation}\label{eq:meetlowerbound2}
\tmeetn \ge \frac{(|E_n|-1)^2}{4|E_n|} .
\end{equation}
All our examples below can be viewed as random walks on graphs,
although we do not use this language for the
examples in Section~\ref{eg:sdt} which include and generalize Theorem~2 of \cite{C:CRW89}.

\subsection{Discrete tori}\label{eg:sdt}
  For $n,d\in\Bbb N$, we consider irreducible
  $(q^{(n)},E_n)$-Markov chains where for $d,n\in \Bbb N$,  
\[
E_n=\big( (-n/2,n/2]\cap
\bZ\big)^d
\]
and $q^{(n)}(x,y)=q^{(n)}(0,y-x)$ for $x\neq y$. Here, the 
difference $y-x$ is read coordinate-wise $\mod n$. By the
assumed symmetry of $q^{(n)}$, the bound
\eqref{eq:meetlowerbound2} applies.

\subsubsection{Nearest-neighbor walk}
Assume $d\ge 2$ and $q^{(n)}(x,y)=(2d)^{-1}$ if $|x-y|=1$ (the
difference is computed $\mod n$ coordinate-wise). Then as $n\lra\infty$,
$\tmixn=O(n^2)$ in all dimensions $d$ (see Theorem 5.5 in \cite{LPW}) and 
\begin{equation}\label{eq:nnrw}
\tmeetn \sim \begin{cases}
\dfrac{1}{2\pi}|E_n|\log |E_n| &\text{if }d=2,\\
G_d|E_n| &\text{if }d\ge 3,
\end{cases}
\end{equation}
where the constant $G_d$ is the expected number of visits to the
origin by a simple symmetric random walk in $\bZ^d$ starting
at the origin (see \cite{C:CRW89}).  Hence, (i) of
Theorem~\ref{thm:main2} holds, and we have the convergence of
voter model densities to the Wright-Fisher diffusion in
\eqref{eq:main1}. This result was first obtained in Theorem~2 of
\cite{C:CRW89}.

\begin{rmk}\label{rmk:rec_trans}
As in Section 13.2.3 in \cite{AF:MC}, we say that the sequence
$\big(q^{(n)},E_n\big)_{n\in \Bbb N}$ is transient if
$\sup_{n\in \Bbb N}\tmeetn /|E_n|$ is finite, and is
  recurrent otherwise. We note that the
asymptotic behavior in \eqref{eq:nnrw} indicates recurrence for $d=2$ and
transience for $d\ge3$. This is consistent with the fact that simple symmetric
random walk on $\Bbb Z^d$ is recurrent if $d=2$ and is
transient if $d\geq 3$.

With this notion in mind, we note that
\eqref{eq:meetlowerbound2} gives the correct asymptotic rate
of growth for $\tmeetn$ for the transient case $d\ge 3$, but not for the
recurrent case $d=2$. \qed
\end{rmk}

\subsubsection{Intermediate-range random walk}\label{eg:int}
We consider the random walks studied in 
\cite{C:CRW2010}, which have range tending to infinity.
Let $(m_n)$ be a sequence of positive integers such that
$m_n<n/2$ for all $n$ and $m_n\to\infty$. 
For any $d\ge 1$, let 
\begin{equation}\label{eq:lambdan}
\Lambda_n=\Lambda^d_n=([-m_n,m_n]\cap\bZ)^d\setminus \{0\},
\end{equation}
and put 
\begin{equation}\label{eq:qlambdan}
q^{(n)}(x,y) = |\Lambda_n|^{-1}\text{ if } y-x \in \Lambda_n
\end{equation}
(again the difference $y-x$ is read mod $n$ coordinate-wise).

\begin{prop}\label{prop:imm1}
Assume $d=2$ and 
\begin{equation}\label{eq:mn2d}
\lim_{n\to\infty}\dfrac{m^2_n}{\log n}=0.
\end{equation}
Then
\begin{align}
&\tmix^{(n)}=O(n^2/m^2_n) \quad\text{ as }n\to\infty,\label{int1}\\
&\liminf_{n\to\infty}\dfrac{\tmeetn}{n^2\log n/m^2_n}>0 .\label{int2}
\end{align}
\end{prop}

Taken together, \eqref{int1} and \eqref{int2} imply
condition (i) of Theorem~\ref{thm:main2}, and so we have the
convergence of voter model densities in \eqref{eq:main1}.  

\begin{proof}[Proof of Proposition~\ref{prop:imm1}]
To obtain \eqref{int1} and \eqref{int2}, we make use of
results from of \cite{C:CRW2010}.  Since conditions
(P1)--(P3) in \cite{C:CRW2010} hold by Proposition~1.1
there,  we deduce from Theorem~1.7 of \cite{C:CRW2010} that
if $\lim_{n\to\infty}s_n/(n^2/m^2_n) =\infty$, then
\[ 
\sum_{x\in E_n} \left|q^{(n)}_{s_n}(0,x)-\pi_{n}(x)\right|
\le n^2\sup_{x\in E_n}\left|q^{(n)}_{s_n}(0,x)-\frac{1}{n^2}\right| \to 0,
\]
which implies \eqref{int1}.

Next, to get (\ref{int2}), we first reduce $\tmeetn$ to a simpler time.
Let $(X_t)$ be the Markov chain on $E_n$ given by
$q^{(n)}$, and
\[
H_0=\inf\{t\geq 0: X_t=0\}
\]
be the
hitting time of 0. Since the difference of two rate-one random walks is a rate-two random walk (see also Proposition
7.1 and Proposition 14.5 of \cite{AF:MC} for a more general
fact), we have
\begin{align}\label{halftime}
\mathbf{E}^{(n)}[M_{U,U'}] =
\frac{\mathbf{E}^{(n)}_{\pi^{(n)}}[H_0]}{2}.
  \end{align}
The limit \eqref{int2} can then be derived from the estimates on
the expectations $\mathbf{E}^{(n)}_{x}[H_0]$ given in
Theorem~1.3 in \cite{C:CRW2010}, but we will use instead the
following simpler argument, which relies on only (6.1) from
\cite{C:CRW2010}.

We now claim
\begin{align}\label{int2-1}
\liminf_{n\to\infty} \dfrac{\mathbf{E}^{(n)}_{\pi^{(n)}}[H_0]}{n^2\log n/m_n^2}\ge
\frac{12}{\pi},
\end{align}
as entails (\ref{int2}) by (\ref{halftime}). For $x\in E_n$ and $\lambda>0$,
let 
\[
G_n(x,\lambda) = \int_0^\infty e^{-\lambda
  s}q^{(n)}_s(0,x)ds. 
\]
By standard Markov chain arguments,
\begin{equation}\label{eq:mctransform}
\mathbf{E}^{(n)}_x[\exp(-\lambda H_0)] =
\dfrac{G_n(x,\lambda)}{G_n(0,\lambda)}.
\end{equation}
Clearly $
\sum_{x\in E_n} \pi^{(n)}(x)G_n(x,\lambda)=\frac{1}{n^2\lambda}
$, 
and thus by \eqref{eq:mctransform}, 
\begin{align}\label{eq:mctransform1}
\ds 
\mathbf{E}^{(n)}_{\pi^{(n)}}[e^{-\lambda H_0}] = \dfrac{(n^2\lambda)^{-1}}{
  G_n(0,\lambda)}.
\end{align}
According to (6.1) of \cite{C:CRW2010} with $t_n= \log n/m^2_n$, we have
\[
\lim_{n\to\infty}\frac{G_n\left(0,\frac{\lambda}{n^2t_n}\right)}{t_n} = \lambda^{-1} + \frac{12}{\pi}.
\]
Applying this fact to (\ref{eq:mctransform1}), we get
\[
\lim_{n\to\infty}\mathbf{E}^{(n)}_{\pi^{(n)}}[e^{-\lambda H_0/n^2 t_n}]= \dfrac{1}{1+\frac{12}{\pi}\lambda}.
\]
We deduce (\ref{int2-1}) from the Skorokhod representation and Fatou's lemma. The proof is complete
\end{proof}

In view of \eqref{int2},
the condition $m^2_n/\log n\to 0$ implies 
$\tmeetn/|E_n|\to\infty$. This means the Markov chain
sequences considered in this example are, like the nearest-neighbor
$d=2$ case, recurrent. Also, although we will not give the
details here, Theorem~\ref{thm:main2} still holds if instead of 
\eqref{eq:mn2d} we consider the (transient) case
in which \eqref{eq:mn2d} is replaced with $\lim_{n\to\infty}m^2_n/\log n=\infty$.

\begin{prop}\label{prop:imm2}
Assume $d=1$ and
\begin{equation}\label{eq:mn1d}
\lim_{n\to\infty}\dfrac{m_n}{\sqrt n} =\infty .
\end{equation}
For $q^{(n)}$ given as in 
$\eqref{eq:qlambdan}$, let $\gap_n$ denote the spectral gap of $q^{(n)}$. Then 
\begin{equation}\label{eq:gapgoal}
\lim_{n\to\infty}n\gap_n=\infty.
\end{equation}
\end{prop}
Assuming Proposition~\ref{prop:imm2} for now, we may write the the second condition in
(ii) of \eqref{thm:main2} in the form
\[
\dfrac{\log(e\vee \tmeetn\pinmax) }{\gap_n\tmeetn}
=
\dfrac{\log[e\vee (\tmeetn/n)] }{(n\gap_n)(\tmeetn/n)} .
\]
By the meeting time bound \eqref{eq:meetlowerbound2},
$\tmeetn/n$ is bounded away from 0, and thus
\eqref{eq:gapgoal} implies that the right-hand side above
tends to 0.  That is, condition (ii) of
Thereom~\ref{thm:main2} holds and we obtain convergence of
voter model densities to the Wright-Fisher diffusion.

For the proof of Proposition~\ref{prop:imm2}, we recall the definition of the bottleneck 
ratio $\Phi_*$ here.
For a reversible Markov chain $(q,E)$ with $q=p-I$ for a probability matrix $p$ with zero diagonal, define 
\begin{equation}\label{eq:Phi}
\Phi(S) = \dfrac{\sum_{x\in S,y\in S^\complement}
\pi(x)q(x,y)}{\pi(S)}, \ S \subset E, 
\end{equation}
and 
\begin{equation}\label{eq:cheegerconst}
\Phi_*(q) = \min\left\{\Phi(S); S\subset E, 
\pi(S)\le \frac12\right\} .
\end{equation}
The inequality we need is
\begin{equation}\label{eq:cheegerbound}
\gap \ge \frac12 \big(\Phi_*(q)\big)^2.
\end{equation}
See Section~13.3.2 in \cite{LPW} for this inequality, and note that $\gap$ is equal to $1-\lambda$ for $\lambda$ being the second largest eigenvalue of $p$.

\begin{proof}[Proof of Proposition~\ref{prop:imm2}]
It is easy to see from the definition (\ref{eq:Phi}) that 
\begin{align}\label{eq:PhiS}
\Phi(S) = \dfrac{|\partial S|}{2m_n|S|},
\end{align}
where $\partial
S = \{(x,y);x\in S, y\in S^\complement, 1\le |x-y|\le m_n
\}$.
For $1\le k\le
n/2$, let $I_k$ be
an ``interval'' of $k$ 
elements in $E_n$:
\[
I_k=\{0,1,\dots,k-1\}. 
\]
A little thought shows that the
minimum of $|\partial S|$ among all $S$ with $|S|=k$ is
obtained by taking $S=I_k$, which implies that
\[
\Phi_*=\min\{\Phi(I_k); 1\le k\le \lfloor n/2\rfloor \}. 
\]
It is easy
to check that if  $m_n<k\le n/2$, then  
\[
|\partial I_k| = 2\sum_{j=1}^{m_n} j = m_n(m_n+1).
\]
Similarly, if $1\le k\le m_n$, then $|\partial I_k|$ is
\[
2\sum_{j=m_n-k+1}^{m_n}j = m_n(m_n+1)-(m_n-k)(m_n-k+1) =
k(2m_n-k+1).
\]
It follows from (\ref{eq:PhiS}) that
\[
\Phi(I_k) = \begin{cases}
\frac{m_n+1}{2k} &\text{if }m_n<k\le n/2,\\
\frac{2m_n-k+1}{2m_n} &\text{if }1\le k\le m_n.
\end{cases}
\]
Taking $k=n/2$, we see that 
\[
\Phi_*\big(q^{(n)}\big)=\Phi\left(I_{\lfloor n/2\rfloor }\right) = \frac{m_n+1}{2\lfloor n/2\rfloor }.
\]
It is now immediate from \eqref{eq:mn1d} and the inequality
\eqref{eq:cheegerbound} that 
\[
\lim_{n\to\infty}n \gap_n \ge \lim_{n\to\infty}\frac{n}{2} \cdot \frac{m_n^2}{4\lfloor n/2\rfloor^2}= \infty,
\]
which completes the proof. 
\end{proof}

Although we will not prove it here, 
the condition \eqref{eq:mn1d} implies that $\tmeetn=O(n)$, which
means that the chain considered in Proposition~\ref{prop:imm2} is transient.

\subsection{Random walk on simple graphs}
We consider in this section graphs which are simple, that is have no loops or
multiple edges, and are connected. The simple
random walk on such a graph $G=({\sf V},{\sf E})$ with vertex set ${\sf V}$ and edge set ${\sf E}$ is the Markov chain $(q,{\sf V})$ with
$q(x,y)=1/\deg(x)$ if $(x,y)$ is an edge for $x\neq y$. Note that  $q$ is
reversible with stationary distribution
\[
\pi(x)=\frac{\deg(x)}{2|{\sf E}|}. 
\]
See \cite{L:RWG} for a survey and the standard terminology of random
walks on graphs.

\subsubsection{Hypercubes}
For $n\in\bN$, take ${\sf V}_n=\{0,1\}^n$ and for $x,y\in {\sf V}_n$ let
$|x-y|=\sum_{j=1}^n|x_i-y_i|$. We draw an edge between any $x,y\in
{\sf V}_n$ with $|x-y|=1$, and obtain the the $n$-dimensional
hypercube, a connected $n$-regular graph. The random walk
$Q$-matrix $q^{(n)}$ on this graph is given by $q^{(n)}(x,y)=1/n$ 
if $|x-y|=1$, and is  irreducible and
symmetric with $\pi(x)\equiv 2^{-n}$. Furthermore, it is known (see
Example~5.15 in \cite{AF:MC}) that  
\[
{\mathbf g}_n=\left(\frac{n}{2}\right)^{-1}\quad\mbox{and}\quad
\tmeet^{(n)}\sim 2^{n-1}\;\;\mbox{as $n\lra\infty$} .
\]
It is easy to see from these facts that (ii) of 
Theorem~\ref{thm:main2} is satisfied.

\subsubsection{Expander graphs} Fix $\alpha\in(0,\infty)$ and
$k\in\Bbb N$ with $k\ge 3$, and take a $(k,\alpha)$-expander family of graphs $(G_n)$ with corresponding random walk
$Q$-matrices $q^{(n)}$. Here as in Section 13.6 of
\cite{LPW}, $(G_n)$ is a graph sequence such that the
number of vertices of $G_n$ tends to infinity, each $G_n$ is
connected and $k$-regular, and satisfies
\[
\Phi_*(q^{(n)}) \geq
\alpha,\quad \forall\;n\in \Bbb N 
\]
(see (\ref{eq:cheegerconst}) for notation).
By \eqref{eq:cheegerbound}, $\liminf \gap_n\ge
\frac12\alpha^2$, and thus the conditions of Corollary~\ref{cor:main} apply.

\subsection{Random walk on general graphs}
We now consider finite graphs without the simplicity condition, nor the connectivity condition.
For such a graph $G$ with vertex set ${\sf V}$, its edge set $\sf E$ is now defined by using an adjacency matrix
\[
A:{\sf V}\times {\sf V}\lra\Bbb Z_+
\]
with $A(x,x)\in \{0,1,2\}$,
so that $A(x,y)$ gives the
number of edges joining $x$ and $y$.
For $x\neq y$, $A(x,y)$ simply gives the number of edges between $x$ and $y$.
In Section~\ref{sec:rgt} below, we consider 
several models of random graphs due to Friedman \cite{F:RG} in which the convention is that $A(x,x)=1$ means a ``half-loop'' at $x$, and
$A(x,x)=2$ means a ``whole-loop'' at $x$.

If we take a sequence $(G_n)$ of such general graphs with $G_n=({\sf V}_n,{\sf E}_n)$ and ${\sf E}_n$ being encoded by $A_n$, then
the $q^{(n)}$-random walk on $G_n$ has $Q$-matrix defined by
\[
q^{(n)}(x,y) = \dfrac{A_n(x,y)}{A_n(x)},\quad x\neq y,
\]
where $A_n(x)=\sum_{y\in {\sf V}_n} A_n(x,y)$. Hence, $q^{(n)}=p^{(n)}-I$, where the $x$-th row of $p^{(n)}$ is obtained by dividing the $x$-th row of $A_n$ by $A_n(x)$. In this case, the second largest eigenvalue of the transition matrix $p^{(n)}$ is different from $1$ if and only if
the graph $G_n$ is connected, and so
the second smallest eigenvalue of $-q^{(n)}$ is different from $0$ if and only if
the graph $G_n$ is connected. 
 
\subsubsection{Random regular graphs}\label{sec:rgt}
The work \cite{F:RG} considers various models of growing
random $k$-regular graphs $(G_n)$ on $n$ vertices (see  the models $\mathcal G_{n,k}$, $\mathcal H_{n,k}$, $\mathcal I_{n,k}$, and $\mathcal J_{n,k}$ there), and each is defined for a large set of admissible degrees $k$. For simplicity, we only consider the model $\mathcal G_{n,k}$ below, although the following discussion applies to other models $\mathcal H_{n,k}$, $\mathcal I_{n,k}$, and $\mathcal J_{n,k}$  in \cite{F:RG} for moderately large admissible degrees $k$ as well.

The random regular graphs $\mathcal G_{n,k}$ are defined for even integers $k$ with $\mathsf V_n=\{1,\cdots,n\}$, and for each $n$ the edge set is given by
\[
{\sf E}_n = \left\{ \big(x,\rho_j(x)\big),\big(x,\rho^{-1}_j(x)\big);
  j=1,\dots,k/2, x\in {\sf V}_n 
\right\},
\]
where $\rho_1,\dots,\rho_{k/2}$ are
i.i.d. permutations of $\{1,\cdots,n\}$ and each $\rho_i$ is chosen uniformly
from the set of $n!$ permutations. Then for any even integer $k$, we have
\[
\lim_{n\to\infty}\P\left(\max_{2\leq i\leq n}|\lambda_i(\mathcal G_{n,k})|
\leq \frac{2\sqrt{k-1}}{k}+\vep\right)=1,\quad \forall\; \vep>0,
\]
where 
\[
1=\lambda_1(\mathcal G_{n,k})\geq \lambda_2(\mathcal G_{n,k})\geq \cdots \geq \lambda_{n}(\mathcal G_{n,k})
\]
are the ordered eigenvalues associated with the normalized adjacency matrix $p^{(n)}$ on $\mathcal G_{n,k}$. More precisely, we have
\begin{align}\label{ineq:RG}
\P\left(\max_{2\leq i\leq n}|\lambda_i(\mathcal G_{n,k})|
> \frac{2\sqrt{k-1}}{k}+\vep\right)\leq \frac{c}{n^{\lceil (\sqrt{k-1}+1)/2\rceil-1}},
\end{align}
where $c$ is a constant. See Theorem 1.1 in \cite{F:RG}, and also \cite{LS:RG} for estimates of mixing times on other random regular graphs.

If we assume in addition that $\mathcal G_{n,k}$, for $n\in \Bbb N$, are independent random graphs, then
it follows from (\ref{ineq:RG}) and the Borel-Cantelli Lemma that for each even $k\geq 12$,
\[
\liminf_{n\to\infty}\gap(\mathcal G_{n,k})>0\quad\mbox{a.s.}
\]
Since the stationary distribution of $q^{(n)}$ is always uniform, 
the sequence of $q^{(n)}$-Markov chains
now satisfies the conditions of Corollary~\ref{cor:main} with probability one (with respect to the randomness of $q^{(n)}$).
We
obtain the convergence of voter model densities \eqref{eq:main1} along $(\mathcal G_{n,k})_{n\in \Bbb N}$ with probability one.

\end{document}